\documentclass[a4paper,11pt,reqno,natbib]{amsart}

\usepackage[utf8]{inputenc}
\usepackage[english]{babel}
\usepackage[T1]{fontenc}    
\usepackage{lmodern}

\usepackage[left=3cm,right=3cm,top=2.5cm,bottom=2.5cm]{geometry}

\usepackage{amsmath}
\usepackage{amssymb}
\usepackage{amsthm}
\usepackage{amsfonts}
\usepackage{bm}
\usepackage{mathrsfs}
\usepackage{mathtools}

\usepackage{graphicx}
\usepackage{tikz}
\usepackage{caption}
\usepackage{enumitem}

 \allowdisplaybreaks \numberwithin{equation}{section}
\begingroup
\theoremstyle{plain}
\newtheorem{theorem}{Theorem}[section]
\newtheorem{proposition}[theorem]{Proposition}
\newtheorem{lemma}[theorem]{Lemma}

\newtheorem{definition}[theorem]{Definition}
\newtheorem{remark}[theorem]{Remark}
\newtheorem{example}[theorem]{Example}
\endgroup

\usepackage[foot]{amsaddr}

\title{Shape Optimization of hemolysis for shear thinning flows in moving domains}

\author{Valentin Calisti$^{\boldsymbol{1}}$, \v{S}\'{a}rka Ne\v{c}asov\'{a}$^{\boldsymbol{2}}$}

\address{
$^{{\boldsymbol{1}} , {\boldsymbol{2}}}$Institute of Mathematics of the Czech Academy of Sciences, 
\v{Z}itn\'{a} 25, 115 67 Praha 1, Czech Republic. 
}
\email{$^{\boldsymbol{1}}$calisti@math.cas.cz, $^{\boldsymbol{2}}$matus@math.cas.cz}
\date{\today}

\newcommand{\pmS}{k}
\newcommand{\pmR}{M}
\newcommand{\pmE}{m}
\newcommand{\pmA}{n}

\newcommand{\hardy}{\mathcal{H}}
\newcommand{\bogov}{\mathcal{B}}
\newcommand{\korn}{\mathcal{K}}
\newcommand{\dn}{3}

\newcommand{\conv}{*}

\newcommand{\Qhad}{\mathfrak{Q}}
\newcommand{\inj}{\hookrightarrow}

\newcommand{\stress}{\mathbf{S}}

\newcommand{\vel}{\mathbf{V}}

\newcommand{\pr}{\pi}
\newcommand{\prn}{\pi^{\pmS}}

\newcommand{\uu}{\boldsymbol{u}}
\newcommand{\extu}{\widetilde{\boldsymbol{u}}}
\newcommand{\uun}{\boldsymbol{u}^{\pmS}}
\newcommand{\extun}{\widetilde{\boldsymbol{u}}^{\pmS}}

\newcommand{\ww}{\boldsymbol{w}}
\newcommand{\hh}{\boldsymbol{h}}
\newcommand{\vv}{\boldsymbol{v}}
\newcommand{\FF}{\boldsymbol{F}}

\newcommand{\hemo}{h}

\newcommand{\sol}{\mrm{div}}
\newcommand{\sym}{\mrm{sym}}

\newcommand{\Gs}{\mathrm{D}}
\newcommand{\dm}{\,d}

\newcommand{\espX}{\underline{\mrm{X}}{\hspace{0.25ex}}}

\newcommand{\Xq}{W^{1,q}_{\div}}

\newcommand{\Vq}{W^{1,q}_{0,\div}}
\newcommand{\Vp}{W^{1,p}_{0,\div}}
\newcommand{\Yq}{L^{q}_{\div}}

\newcommand{\Hq}{L^{q}_{0,\div}}

\newcommand{\SOinf}{\mrm{I}_{\SHPfun}}

\def \div {\mathop {\rm div}\nolimits} 
\newcommand{\bd}[1]{\boldsymbol{#1}} 
\newcommand{\mrm}[1]{\mathrm{#1}}
\newcommand{\supp}{\mathop {\rm supp}\nolimits} 
\newcommand{\dist}{\mathop {\rm dist}\nolimits} 
\newcommand{\Lip}{\mathop {\rm Lip}\nolimits} 
\newcommand{\gd}{\nabla}
\newcommand{\Ksub}{\subset\subset} 
\newcommand{\lV}{\lVert}
\newcommand{\rV}{\rVert}
\newcommand{\lv}{\lvert}
\newcommand{\rv}{\rvert} 
\newcommand{\NN}{\mathbb N} 
\newcommand{\RR}{\mathbb R}  
\newcommand{\nC}{\mathcal C} 
\newcommand{\D}{\mathcal D}
\newcommand{\nH}{\mathcal H}
\newcommand{\Opn}{\mathcal O}
\newcommand{\nP}{\mathcal P}
\newcommand{\R}{\mathcal R} 
\newcommand{\V}{\mathcal V}  
\newcommand{\X}{\mathcal X}
\newcommand{\BA}{\mathbf A}
\newcommand{\bB}{\mathbf B}
\newcommand{\Bg}{\mathbf g}
\newcommand{\BT}{\mathbf T}
\newcommand{\had}{D} 
\newcommand{\SHPfun}{\mathcal{J}}
\newcommand{\shpFUN}{j}

\usepackage[backend=bibtex, 
		   giveninits, 
		   doi=false,
		   isbn=false,
		   url=false,
		   maxbibnames=9,
		   maxcitenames=4
		   ]{biblatex}
\begin{document}

\maketitle

\begin{abstract}
We consider the $3$D problem of shape optimization of  blood flows in moving domains. 
Such a geometry is adopted to take into account the modeling  of rotating systems and blood pumps for instance. 
The blood flow is described by generalized Navier-Stokes equations, in the particular case of shear-thinning flows. 
For a sequence of converging moving domains, we show that a sequence of associated solutions to blood equations converges to a solution of the problem written on the limit moving domain.  
Thus, we extended the result given in (Soko\l{}owski, Stebel, 2014, in \textit{Evol. Eq. Control Theory}) for $q \geq 11/5$, 
to the range $6/5< q < 11/5$, 
where $q$ is the exponent of the rheological law. 
This shape continuity property allows us to show the existence of minimal shapes for a class of functionals depending on the blood velocity field and its gradient. 
This allows to consider in particular the problem of hemolysis minimization in blood flows, namely the minimization of red blood cells damage.

\end{abstract}


{\small
\tableofcontents
}

%#####################################################################
%#####################################################################
\section{Introduction}

For patients suffering from heart failure, blood pumps may be essential,  
but there are risks of complications associated with this type of device, 
such as bleeding, thrombosis and infections \cite{thamsen2015numerical}. 
The other detrimental consequence we are concerned with in this article is \emph{hemolysis}, 
which refers to the destruction of red blood cells (RBCs), leading to a release of free hemoglobin in the blood plasma.    
For instance, this can be toxic and provoke renal failure.

The natural question that arises is how to minimize hemolysis. 
Numerous 
studies have focused on the numerical optimization of design parameters of blood flows or pumps. 
We refer among others 
to \cite{abraham2005shape-unstead}  
where a Carreau-Yasuda model is considered, 
and to \cite{alonso_silva2021optopo_blood} concerning a Tesla-type pump and the minimization of hemolysis. 
To generalize this approach, we address the theoretical shape optimization of these pumps, enabling us to broaden the set of admissible solutions.   
Since the state of hemolysis in the blood depends directly on the gradient of the blood velocity field (see Example~\ref{thm:ex:hemolysis}), 
the present work therefore focuses on the shape optimization of blood flows in moving domains.

The main obstacle to be overcome is the following:  
the non-Newtonian behavior of blood has important effects on blood flows and should be taken into account (see, e.g. \cite{l_m_z2008shear-thinning}).  
Indeed, human blood is mainly composed of plasma  
and RBCs,  
and their interaction yields a \emph{shear-thinning} behavior,  
that is the \emph{viscosity} $\nu$ decreases when the shear rate increases.  
To take this effect into account, we assume that the blood viscosity depends non-linearly  
on the symmetric gradient of the velocity field $\Gs \uu$.
Namely, we assume that the \emph{viscous stress tensor} $\stress ( \Gs \uu ):= \nu (\Gs \uu) \Gs \uu$ is given by:
\begin{equation}
\label{eq:def:stress}
\stress ( \Gs \uu )  := \left( 1 + \lvert \Gs \uu \rvert \right)^{q-2}  \Gs \uu , 
\quad 
\text{for } 6/5 < q < 2, 
\quad 
\text{and } 
\quad 
\Gs \uu  := \frac{1}{2} ( \gd \uu + \gd \uu ^\top ).
\end{equation}
We first note that for $q=2$, the classical Newtonian viscous stress tensor is retrieved, for a constant 
viscosity $\nu = 1$. 
Fluids satisfying \eqref{eq:def:stress} 
with $q>2$ are called \emph{shear thickening fluids}  
(e.g. cornflour batter). 
Blood, on the contrary, is a \emph{shear thinning fluid} with $q<2$  
(see, e.g. \cite{farina_et-al2016non-newt}, where $q\simeq 1.22$).  
For this type of rheological law, we have no uniqueness results for weak solutions, which complicates the study of the shape sensitivity of the problem.

To the best of our knowledge, the
shape optimization analysis of shear-thinning fluids has never been addressed. 
In \cite{sokolowski_stebel_2014shapo_nonnewt_t-dep}, the existence of shapes minimizing generic functionals is given for non-Newtonian fluids with 
$q \geq 11/5 $ in moving domain. 
Here, we extend this result to the range $6/5 < q < 11/5$ (Section~\ref{sec:continuity:blood}), relying on the existence result provided in \cite{nagele_ruzicka2018}.  
This leads to the proof (Section \ref{sec:exist-min}) of the main result written in Theorem~\ref{thm:main-result}.
After a presentation of the geometry and the blood model in Subsection \ref{sec:hemo-model}, 
we introduce the shape optimization problem in Subsection~\ref{sec:shapo:def}, together with the main result. 
In particular, this allows us to show that there exists optimal shapes for the hemolysis minimization problem in moving domains (see Example~\ref{thm:ex:hemolysis}).

%*********************************************************************
%*********************************************************************
\subsection{A blood model in moving domains}
\label{sec:hemo-model}

Let $\had$ be a bounded open subset of $\RR^{\dn}$, and $\Omega$ an open set such that $\Omega \Ksub \had$. 
Let $\vel$ be a smooth solenoidal  velocity field such that $\supp (\vel) \subset\subset [0,T] \times \had$.
We denote by $\bd{\varphi}$ the unique solution to:  
\begin{equation}
\label{eq:def:transfo-assoc-velocity}
\left\lbrace
\begin{aligned}
\partial_t \bd{\varphi} (t , x) &= \vel ( t , \bd{\varphi} (t , x) )  && \text{ for }(t,x) \in [0,T] \times \had , \\
\bd{\varphi} ( 0 , x ) & = x && \text{ for }  x \in \had ,
\end{aligned}
\right.
\end{equation}
and we define (see Figure \ref{fig:mov-dom}): 
\begin{equation}
\label{eq:def:Omegat}
\Omega_t := \bd{\varphi} ( t , \Omega ) , \quad 
Q := \bigcup_{t \in (0,T)} \lbrace t \rbrace  \times \Omega_t , \quad
\Sigma := \bigcup_{t \in (0,T)} \lbrace t \rbrace  \times \partial\Omega_t .
\end{equation}
We also define the time-space cylinder containing any moving domain $Q$ defined as above, that we
call \emph{hold-all cylinder}:
\begin{equation*}
\Qhad := (0,T) \times \had.
\end{equation*}

\begin{figure}[h]
\begin{center}
\begin{tikzpicture}[scale=0.8]

\begin{scope}[scale=0.75, xshift=-4.3cm, yshift=-2.7cm]
\def\WiLi{0.2}
\def\haut{4}
\draw[->,dash pattern=on 2pt off 2pt, line width = \WiLi mm] 
(-1,0) -- (9,0)  ; 
\draw[dash pattern=on 2pt off 2pt, line width = \WiLi mm] 
(8,0) -- (8,5)  ; 
\draw[dash pattern=on 2pt off 2pt, line width = \WiLi mm] 
(8,5) -- (0,5) ; 
\draw[dash pattern=on 2pt off 2pt, line width = \WiLi mm] 
(0,5) -- (0,0) ; 
\draw[->,dash pattern=on 2pt off 2pt, line width = \WiLi mm] 
(-0.5,-0.5) -- (-0.5,6)  ; 
\draw[line width = 1.3*\WiLi mm, color=white] 
(-0,0) -- (8,0)  ; 
\draw[dash pattern=on 4pt off 2pt, line width = 1.8*\WiLi mm, color=blue] 
(-0,0) -- (8,0)  ; 
\draw[line width = 2*\WiLi mm] 
(2,0) to (5,0) ; 
\draw[line width = 2*\WiLi mm] 
(6,5) to (3,5) ; 
\draw[dash pattern=on 4pt off 2pt, line width = 2*\WiLi mm] 
(5,0) to[out=100, in=250] (6,5) ; 
\draw[dash pattern=on 4pt off 2pt, line width = 2*\WiLi mm] 
(3,5) to [out=240, in=40] (2.5,2.5) to[out=220, in=90] (2,0)  ; 	
\draw[line width = 2*\WiLi mm] 
(-0.7,5) to   (-0.3,5)   ; 	
\end{scope}	

% SCALE =0.75
\draw (-3.6,2.8) node{$t\in \RR$} ;
\draw (4,-1.95) node{$\RR^3$} ;
\draw (-0.3,-2.3) node{$\Omega$} ;
\draw (1.2,-2.3) node{\color{blue}$\had$} ;
\draw (0.2,2) node{$\Omega_T$} ;
\draw (0.1,0) node{$Q$} ;
\draw (-4.2,-2) node{$0$} ;
\draw (-4.2,1.75) node{$T$} ;

\draw (2.4,1.4) node{$\Qhad$} ;
\end{tikzpicture}
\caption{Moving domain $Q$.}
\label{fig:mov-dom}
\end{center}
\end{figure}
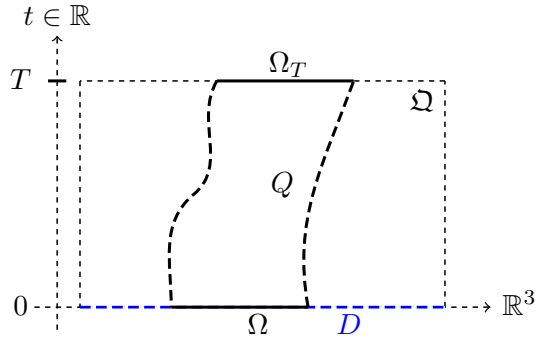
We need to describe a model of blood flows. 
As it was explained in the first part of the introduction, we shall describe blood as an incompressible non-Newtonian fluid, so that its motion satisfies  the following generalized Navier-Stokes 
equations:
\begin{equation}
\label{pbm:blood_flow}
\left\lbrace
\begin{aligned}
\frac{\partial \uu }{\partial t} + \uu \cdot \nabla \uu - \div ( \stress ( \Gs \uu ) ) + \gd \pr &= \bd{f}   && \text{ in } Q , \\
\div \uu &= 0   && \text{ in } Q , \\
\uu  &= \vel   && \text{ on }  \Sigma , \\ 
\uu(0) &= \uu_0   && \text{ in }  \Omega  ,
\end{aligned}
\right.
\end{equation}
where the non-Newtonian stress tensor $\stress ( \Gs \uu )$ is given in \eqref{eq:def:stress}, $\uu_0$ is the initial velocity field, $\bd{f}$ the body force, and $\pi$ the pressure. 

\smallskip

\textit{\underline{Bibliographical remarks:}}
Concerning the cylindrical case for Problem \eqref{pbm:blood_flow}-\eqref{eq:def:stress}, the existence of weak solutions was  proved in \cite{lions1969qques_methodes, ladyzhenskaya1969theory_viscous-incomp-flow} 
for $q \geq 11/5$, 
providing uniqueness as long as $q >5/2$. 
The existence result 
was improved to $q > 9/5$
for 
space periodic boundary conditions,  
together with 
uniqueness for $q \geq 11/5$ in \cite{malek_necas_ruzicka1996weak_meas_sol_evolPDE}.  
Uniqueness for $q <  11/5$ remains an open problem. 
For the homogeneous Dirichlet boundary conditions,    
existence was provided in \cite{wolf2007extist_wsol} for $q >8/5$ with the $L^\infty$-truncation method, and  
in \cite{diening_ruzicka_wolf2010existence} for $q >6/5$ with the Lipschitz-truncation method.  
See also
\cite{diening_ruzicka_wolf2010existence}  for further references on this topic.   
Concerning the non-cylindrical case -- moving domains --,  
existence  was proved in  \cite{sokolowski_stebel_2014shapo_nonnewt_t-dep} for $q \geq 11/5 $ with a penalization method,   
and 
extended to  
$q > 6/5$  
in \cite{nagele_ruzicka2018} with the solenoidal 
Lipschitz-truncation method.

\smallskip

Having modeled the blood flows, we can now define the shape optimization problem, and write the existence of minima, which is the main subject of this article.

%*********************************************************************
%*********************************************************************
\subsection{Shape optimization problem and main result}
\label{sec:shapo:def}

The shape optimization of Navier-Stokes 
equations 
has been widely addressed since the late 90s. 
For references on 
shape sensitivity  
of unsteady incompressible case, see among others \cite{fischer_lindemann_ulbrich2_2017frechet}, 
tackling domain variations of low regularity. 
For incompressible Navier-Stokes  
in moving domains, the shape differentiability is studied in 
\cite{moubachir_zolesio2006moving-shape-analys_control}, 
and in \cite{sokolowski_stebel_2011shape-sensitiv_unsteady} 
for shear-thickening fluids in 2D.

\smallskip

We investigate the following family of shape optimization problems: 
\begin{equation}
\label{eq:pbm:minimization:abstract}
\begin{aligned}
& \text{Find } ( \Omega^* , \vel^* ) \in  \Opn \times \V \text{ such that: }  \\
& \SHPfun ( \Omega^*  ,   \vel^*   )  \leq \SHPfun ( \Omega  ,   \vel   ) , \quad \forall ( \Omega , \vel ) \in  \Opn \times \V , 
\end{aligned}
\end{equation}
where $\SHPfun ( \Omega  ,   \vel   )$ is a shape functional which depends on the solution $\uu$ of blood problem \eqref{pbm:blood_flow}-\eqref{eq:def:stress} defined on the domain $Q$ associated to $( \Omega  ,   \vel   ) \in \Opn \times \V$.  
Here $\Opn $ and $ \V$ are respectively the classes of admissible initial domains $\Omega$ and velocity fields $\vel$, which will be specified later in Section \ref{sec:shapo:setting}.

Let $( \Omega,   \vel   ) \in \Opn \times \V$, we define $\nP ( \Omega,   \vel   ) $ as the set of solutions to the blood problem \eqref{pbm:blood_flow}-\eqref{eq:def:stress}.   
We will see that in the framework of non-Newtonian fluids where the stress tensor is expressed by \eqref{eq:def:stress} for some $q>6/5$, we don't have the uniqueness of the solutions.
Thus, we define the shape functional as follows:
\begin{equation}
\label{eq:def:SHPfun}
\SHPfun ( \Omega  ,   \vel   ) := \inf_{ \uu  \in \nP ( \Omega  ,   \vel   )} \lbrace  \shpFUN ( \Omega  ,   \vel ; \uu  ) \rbrace ,
\end{equation}
where the following assumptions are required for the functional $\shpFUN ( \Omega  ,   \vel ; \uu   )$.

\smallskip

\begin{enumerate}[label=(J.\arabic*), leftmargin=8ex,
%itemsep=1.5ex
]
\item  \label{item:assump:j:1}
For any $( \Omega  ,   \vel   )\in \Opn \times \V$, and any $  \uu   \in \nP ( \Omega  ,   \vel   )$,  
it holds 
$ 
0 \leq \shpFUN ( \Omega  ,   \vel ; \uu   ) 
$. 
\item  \label{item:assump:j:2} 
The functional $\shpFUN$ is lower semicontinuous.
Namely, let $(\Omega^\pmS , \vel^\pmS ) \longrightarrow ( \Omega , \vel )$ be a converging sequence in $\Opn \times \V$, and let $  \uun   \in \nP (\Omega^\pmS , \vel^\pmS )$ be any associated sequence of solution converging to some $  \uu   \in \nP (\Omega , \vel )$ (the definitions of the respective convergences are given in Definition~\ref{def:cvg:dom-veloc} and Theorem~\ref{thm:shapo-cont:blood}). Then:
\begin{equation*}
\shpFUN ( \Omega  ,   \vel  ; \uu    ) 
\leq \liminf_{\pmS \rightarrow \infty}
\shpFUN ( \Omega^\pmS  ,   \vel^\pmS ; \uun   ) .
\end{equation*}
\end{enumerate}

First, we can notice that from assumption \ref{item:assump:j:1}, $\SHPfun ( \Omega  ,   \vel   )$ is well-defined. 
Then we will see in Section \ref{sec:exist-min} that assumptions \ref{item:assump:j:1} and \ref{item:assump:j:2} are used to prove existence of a minimum for Problem \eqref{eq:pbm:minimization:abstract}, using the \textit{Direct Method of Calculus of Variations}.  
This main result of the present article is given below.

\begin{theorem}
\label{thm:main-result}
Let $\SHPfun ( \Omega  ,   \vel   ) $ be the shape functional defined in \eqref{eq:def:SHPfun}, where $\shpFUN$ satisfies assumptions \ref{item:assump:j:1}-\ref{item:assump:j:2}.   
Then, the shape optimization problem \eqref{eq:pbm:minimization:abstract} admits at least one solution $( \Omega^* , \vel^* ) $ in the class of admissible domains and velocities $ \Opn \times \V$ given by \eqref{eq:def:class:O}-\eqref{eq:def:class:V}.  
\end{theorem}

\textbf{\textit{\underline{Outline:}}}
For proving Theorem~\ref{thm:main-result},   
the article is organized as follows.
In Section~\ref{sec:settings}, we present the shape optimization settings of the geometry and the model, we consider a minimizing sequence of moving domains $((\Omega^\pmS , \vel^\pmS ))_{\pmS \geq 0}$, and see that we can extract a converging subsequence. 
Then Section~\ref{sec:continuity:blood} is devoted to proving the convergence of associated sequences of blood solutions.  
The proof of Theorem~\ref{thm:main-result} is finally written in Sections \ref{sec:exist-min}. 
Appendix \ref{sec:density}, for its part, contains the proof of the density result from  Proposition~\ref{thm:density:H1L2capLqVq}, required for the shape continuity of blood flows. 
Before starting, let us present a concrete example of a functional that motivated the present study: hemolysis minimization.

\begin{example}[Hemolysis minimization]
\label{thm:ex:hemolysis}
The primary motivation for this work is to find shapes that minimize RBCs damage in blood flows.  
The \emph{hemolysis index}, denoted by $\hemo (\uu)$, is a scalar field introduced to measure RBCs damage, and can be defined in several ways.  
Different models stand, depending on whether the \emph{viscous stress}  or the \emph{strain} is considered as a source of hemolysis production (stress/strain-based models), and depending on the chosen scalar measure of these tensors. 
The original formulation of the hemolysis index is the following (see \cite{yu2017review} for a~review):  
\begin{equation*}
\hemo ( \uu )  = C  \left\lv \stress( \Gs \uu ) \right\rv^\alpha t^\beta ,
\end{equation*}
where $t$ is the time exposure, 
$\stress( \Gs \uu )$ is the viscous stress tensor defined in \eqref{eq:def:stress}, 
and  $C$, $\alpha$ and $\beta$ are positive constants given by the experimental literature. 
Let us define for any $1 \leq  r < \infty$:
\begin{align*}
\shpFUN_r ( \Omega  ,   \vel  ; \uu    ) 
&:= \lV \hemo ( \uu ) \rV_{L^rL^r (t)}^r 
:= \int_0^T \int_{\Omega_t} \lv \hemo ( \uu ) \rv^r \dm x \dm t ,
\end{align*} 
and the related shape functionals: 
$
\SHPfun_r ( \Omega  ,   \vel   ) := \inf_{ \uu \in \nP ( \Omega  ,   \vel   )} \lbrace  \shpFUN_r ( \Omega  ,   \vel ; \uu   ) \rbrace 
$.  
We straightforwardly notice that $\hemo ( \uu )$ is at least as regular as $\left\lv \stress( \Gs \uu ) \right\rv^\alpha$ because $\beta >0$ so that $t^\beta \in L^\infty (\Qhad)$. 
Let $(\Omega^\pmS , \vel^\pmS ) \longrightarrow ( \Omega , \vel )$, 
and let $( \uun )_{\pmS \geq 0}$ (resp. $\uu$) be the associated solutions to Problem \eqref{pbm:blood_flow}-\eqref{eq:def:stress} written on $(\Omega^\pmS , \vel^\pmS ) $ (resp. $(\Omega , \vel) $). 
It will be shown 
-- see \eqref{eq:weaklim:tild:w:5} and 
Section~\ref{sec:identif-stress} --
that there exists a subsequence s.t. 
$\stress ( \Gs \extun ) \rightharpoonup \stress ( \Gs \extu )$ weakly in $L^{q'} ( \Qhad )$ and a.e. in $\Qhad$, 
where $\extun$ (resp. $\extu$) denotes the extension of $\uun$ (resp. $\uu$) to $\Qhad$ by $\vel^\pmS$ (resp. $\vel$),  
and $q' \in (2,6)$ is the conjugate of the rheological exponent $q$ in \eqref{eq:def:stress}. 
Thus for any $\alpha \leq q'$,
and any $1 \leq  r \leq q' / \alpha$, 
we have by definition that $\shpFUN_r ( \Omega  ,   \vel  ; \uu    )$ is well defined and satisfies Assumption \ref{item:assump:j:1}. 
From the weak and a.e. convergences written above, it also holds $\hemo (  \extun ) \rightharpoonup \hemo (   \extu )$ weakly in $L^{r} ( \Qhad )$. 
By definition of the extensions, we have $\shpFUN_r ( \Omega  ,   \vel  ; \uun    ) 
= \int_0^T \int_{\had} \lv \hemo ( \extun ) \rv^r \dm x \dm t 
-  \int_0^T \int_{\had \setminus \Omega_t} \lv \hemo ( \vel^{\pmS} ) \rv^r \dm x \dm t $. 
From the strong convergence  $\vel^{\pmS} \rightarrow \vel$ (see Definition~\ref{def:cvg:dom-veloc}), 
and by lower semicontinuity of the norm for the weak convergence, Assumption \ref{item:assump:j:2} is satisfied as well. 
Then, from Theorem \ref{thm:main-result}, there exists $(\Omega^{*,r} , \vel^{*,r})  \in \Opn \times \V $ which minimizes $\SHPfun_r$, that is to say which minimizes the $L^r$-norm of the hemolysis index. 
It has to be noticed that this example fits to empirical values, with for instance $q \simeq 1.22$ (see \cite{farina_et-al2016non-newt}) and $\alpha \simeq 2.42$ (or $\alpha \simeq 1.99$, see \cite{yu2017review} for other values and references).

\end{example}

%#####################################################################
%#####################################################################
\section{Shape optimization settings}
\label{sec:settings}

First,  we define the classes of admissible initial domains and velocity fields 
and the associated notion of convergence.  
As is standard practice in shape optimization, these classes are chosen to  
enjoy compactness properties, and  to be  
regular enough for the resolution of the boundary value problems.  
Then, we introduce the functional framework for moving domains: the \emph{generalized Bochner spaces}.  
From this, we define what a weak solution is for the blood  problem.

%*********************************************************************
%*********************************************************************
\subsection{Optimization settings: moving domains}
\label{sec:shapo:setting}

Let $\had$ be a bounded open subset of $\RR^{\dn}$, and $c_{\Opn},c_{\V}>0$ be two   
constants. 
Denoting by $\nC ( c_{\Opn} )$ the class of open subsets of $\RR^{3}$ having the \emph{$c_{\Opn}$-cone property} (see \cite[Def. 2.4.1]{henrot_pierre2018shapo_geom}),  
we define the class of admissible initial domains:  
\begin{equation}
\label{eq:def:class:O} 
\Opn = \lbrace 
\Omega \in \nC ( c_{\Opn} ) 
\mid \overline{\Omega} \subset \had  
\rbrace ,
\end{equation}
and the class of admissible velocity fields:% 
\begin{equation}
\V = 
\lbrace \vel \in C^{1,1} ( [0,T] , C^{1,1}_{0,\sol} (D)^3  ) 
\mid 
\lV \vel \rV_{ C^{1,1}   
} \leq c_{\V} \rbrace ,
\label{eq:def:class:V}
\end{equation}
where  
$ C^{1,1} ( [0,T] , C^{1,1}_{0,\sol} (D)^3  ) :=
\lbrace \vel \in C^{1,1} ( [0,T]\times \overline{D} )^3  
\mid 
\supp \vel  \subset [0,T] \times D , \ 
\div_{\bd{x}} \vel = 0 \rbrace $.

Let us motivate these definitions. 
First, it is equivalent for $\Omega$ to have the cone property and to have a Lipschitz boundary \cite[Thm. 2.4.7]{henrot_pierre2018shapo_geom}, 
which is enough to solve the PDEs in this article. 
Secondly, for $\vel \in \V$, the associated flows $\bd{\varphi}$ given by \eqref{eq:def:transfo-assoc-velocity} are $C^1$-diffeomorphisms \cite[Chap. 4, Thm. 4.4]{delfour_zolesio2011shapesNgeom},  
so that each layer $\Omega_t$ has also a Lipschitz boundary \cite[Thm. 4.1]{hofmann_mitrea_taylor2007geom-dom}. 
For $\vel$, we need uniform $C^{1,1}$ regularity 
for compactness reason, see below.

Now let $((\Omega^{\pmS} , \vel^{\pmS}))_{{\pmS} \geq 0}$ be a sequence in $\Opn \times \V$. 
For any $ \pmS \geq 0$, we define the transformation $\bd{\varphi}^\pmS$ as the unique solution to problem \eqref{eq:def:transfo-assoc-velocity} associated to $\vel^\pmS$. 
Thus we define as in \eqref{eq:def:Omegat}:
\begin{equation}
\label{eq:def:Omegat-n}
\Omega_t^\pmS := \bd{\varphi}^\pmS ( t , \Omega^\pmS ) , \quad 
Q^\pmS := \bigcup_{t \in (0,T)} \lbrace t \rbrace  \times \Omega_t^\pmS , \quad
\Sigma^\pmS := \bigcup_{t \in (0,T)} \lbrace t \rbrace  \times \partial\Omega_t^\pmS .
\end{equation}
\begin{definition}\label{def:cvg:dom-veloc}
We say that $((\Omega^{\pmS} , \vel^{\pmS}))_{{\pmS} \geq 0}$ converges to $(\Omega , \vel)$ in $\Opn \times \V$, and we will simply write $(\Omega^\pmS , \vel^\pmS ) \longrightarrow ( \Omega , \vel )$, if the two following convergences hold: 
\begin{enumerate}%[itemsep=1.5ex]
\item \label{def:cvg:dom-veloc:1}
$\Omega^\pmS \overset{\nH}{\longrightarrow} \Omega$ in Hausdorff topology (see \cite[Sec. 2.2.3]{henrot_pierre2018shapo_geom} for a definition),

\item \label{def:cvg:dom-veloc:3} 
$\vel^\pmS \longrightarrow \vel$ in $C^{1} ( [0,T] \times \overline{\had} )$.
\end{enumerate}
\end{definition}

From this definition  we can state the following  classical compactness result.  

\begin{proposition} 
\label{thm:compactness:geom}
Let $((\Omega^\pmS , \vel^\pmS))_{\pmS \geq 0} \subset \Opn \times \V$. 
Then there exists $(\Omega^* , \vel^*) \in \Opn \times \V$, and a subsequence  $((\Omega^{\pmS'} , \vel^{\pmS'}))_{\pmS' \geq 0}$ such that $(\Omega^{\pmS'} , \vel^{\pmS'} ) \longrightarrow ( \Omega^* , \vel^* )$ as $\pmS' \rightarrow +\infty$. 
\end{proposition}

The compactness of $\Opn$ for the Hausdorff topology is well-known \cite[Thm. 2.4.10]{henrot_pierre2018shapo_geom}, 
and for $(\vel^{\pmS})_{\pmS \geq 0}\subset \V$, the uniform bound $\lV \vel^{\pmS} \rV_{C^{1,1}} \leq c_{\V}$ and 
Arzel\`{a}-Ascoli Theorem provide a subsequence 
$\vel^{\pmS'} \longrightarrow \vel$ in $C^1$-norm, and it can be additionally show that actually $\vel \in \V$.

As a consequence, we have the following important properties (see \cite{sokolowski_stebel_2014shapo_nonnewt_t-dep}).
Let $(\Omega^\pmS , \vel^\pmS ) \longrightarrow ( \Omega , \vel )$ in $\Opn \times \V$. 
Then for every cylinder $[T_1 , T_2] \times \overline{K} \subset Q$, where $K$ is an open subset $K\Ksub \had$ and $0<T_1<T_2<T$, there exists $\pmS_0\in\NN$ such that for every $\pmS \geq \pmS_0$:  
$ 
[T_1 , T_2] \times \overline{K} \subset Q^\pmS 
$.   
We give a slightly more general result, that we will need later on to consider compactly supported test functions for writing the weak formulation of blood flows for a sequence of converging moving domain $(( \Omega^\pmS , \vel^\pmS ))_{\pmS\geq 0}$. Let $\epsilon >0$, we define the two following open sets:
\begin{equation}
\Omega_{(\epsilon)}  
:= \left\lbrace 
x \in \Omega   \;\big\vert\;  \dist (x , \partial\Omega) > \epsilon 
\right\rbrace  ,  
\quad
\text{and}
\quad
\Omega^{(\epsilon)} 
:= \left\lbrace 
x \in \Omega   \;\big\vert\;  \dist (x , \partial\Omega) < \epsilon 
\right\rbrace  .
\label{eq:def:int-dom}
\end{equation}

\begin{lemma}
\label{thm:lem:compact-in-Q}
Let $((\Omega^\pmS , \vel^\pmS))_{\pmS \geq 0} $ be a sequence converging to $(\Omega , \vel)$ in $\Opn \times \V$.
Let $Q_K \subset \underline{Q} $ be a compact subset, where $\underline{Q}$ is defined by:
$
\underline{Q} :=  
\bigcup_{t \in [0,T]} \lbrace t \rbrace  \times \Omega_t 
$.
Then, there exists an integer $\pmS_0$, such that for any $\pmS \geq \pmS_0$, we have
$
Q_K \subset \underline{Q}^\pmS :=  \bigcup_{t \in [0,T]} \lbrace t \rbrace  \times \Omega^\pmS_t 
$.
\end{lemma}

\begin{proof}
Let  $\varepsilon , \delta >0$, we start defining the sets $K_\delta := \Omega_{(\delta)} $,  
$
\underline{Q}_{K_\delta}:= 
\bigcup_{t \in [0,T]} \lbrace t \rbrace  \times 
\bd{\varphi}(t ,K_\delta )  
$,  
and  
$
\underline{Q}_\varepsilon:=  \bigcup_{t \in [0,T]} \lbrace t \rbrace  \times \Omega_{t(\varepsilon)}
$,  
recalling that $\Omega_{t(\varepsilon)}$ is defined by \eqref{eq:def:Omegat} and \eqref{eq:def:int-dom}.
Because $Q_K$ is compact and  $Q_K \subset \underline{Q} $, there exists $\varepsilon >0$ such that $Q_K \subset \underline{Q}_\varepsilon$.
Let $\delta < \varepsilon / c_{\V}$.  
In view of $\lv \bd{\varphi}(t,X) - \bd{\varphi}(t,Y) \rv \leq c_{\V} \lv X - Y \rv$ holding $\forall t\in [0,T]$, $\forall X,Y \in \Omega$,  
one can check that $\bd{\varphi}^{-1}(t, \Omega_{t(\varepsilon)} ) \subset K_\delta$, and then 
$
\bd{\varphi}(t,\bd{\varphi}^{-1}(t, \Omega_{t(\varepsilon)} ) ) \subset \bd{\varphi}(t ,K_\delta )
$,  
so that finally $\underline{Q}_\varepsilon \subset \underline{Q}_{K_\delta}$.
Then we can show that there exists $N_0 \in \NN$ such that for all $\pmS  \geq N_0$, 
it holds $\underline{Q}_{K_\delta} \subset \underline{Q}^{\pmS}_{K_{\delta/2}}  := \bigcup_{t \in [0,T]} \lbrace t \rbrace  \times 
\bd{\varphi}^\pmS(t ,K_{\delta/2} )$ 
(from the convergence  
$\vel^\pmS \longrightarrow \vel$ 
in $C^{1} ( [0,T] \times \overline{\had} )$).  
Finally, from the Hausdorff convergence  
$\Omega^\pmS \longrightarrow \Omega$, 
there exists a $N_1 \in \NN$ such that for all $\pmS \geq N_1$, the inclusion $\overline{K_{\delta/2}} \Ksub \Omega^\pmS$ holds, which implies $\underline{Q}^\pmS_{K_{\delta/2}} \subset \underline{Q}^\pmS$, and finishes the proof by setting $\pmS_0 := \max \lbrace N_0 , N_1 \rbrace  $. 
\end{proof}

%*********************************************************************
%*********************************************************************
\subsection{Functional settings}
\label{sec:func:setting}

We recall the definition of \emph{generalized Bochner spaces} (see \cite{nagele_ruzicka2018}), for a time-dependent family of separable Banach spaces $(B(t))_{t\in [0,T]}$ such that $B(t) \hookrightarrow L^1 ( \Omega_t )$, and $1 \leq r \leq \infty$. 
We define: 
\begin{multline*}
L^r ( (0,T) , B(t) ) := \Big\lbrace
u \in L^1( Q ) \;\Big\vert\;  u(t) \in B(t) \text{ for a.e. } t \in ( 0,T ) , 
\  t\mapsto \lV  u(t) \rV_{B(t)} \in L^r ((0,T))
\Big\rbrace 
\end{multline*}
and we define its norm as follows:
\begin{equation*}
\label{eq:def:norm:gnrl-Bochner}
\lV u \rV_{L^r_{0,T} B(t)} := \bigg( 
\int_0^T \lV u \rV_{B(t)}^r \dm t 
\bigg)^{\frac{1}{r}} .
\end{equation*} 
Let $u \in B ( t ) $, and $v $ belong to the dual space $B(t)'$.
In the following, we will write the duality product between $u$ and $v$ as: 
$\langle v , u \rangle_{B(t)}$. 
In the particular case where $u,v \in L^2 (\Omega_t)$, we will write:  
$ 
( v,u  )_{L^2 (\Omega_t)} := \int_{\Omega_t} v (x) u (x)  \dm x 
$.  
For $d$-dimensional vector functions $\uu \in L^r ( (0,T) , B(t)^d )$ and $\vv (t) \in (B(t)^d)'$, the norm and duality product will simply be written $\lV \uu \rV_{L^r_{0,T} B(t)}$ and $\langle \vv (t) , \uu (t) \rangle_{B(t)}$. 
For the sake of readability, 
when no ambiguity, these spaces will simply be denoted as  
$
L^r B(t)^d := L^r ( (0,T) , B(t)^d ) 
$.

Now we define the \emph{generalized time derivative} of $u \in L^r B(t)$ as follows: $\partial_t u \in \D ' (Q) $ such that 
$
\langle \partial_t u , \eta \rangle := -\langle  u , \partial_t \eta \rangle
$
for all $\eta \in \D (Q)$, 
recalling  that $L^r B(t) \inj L^1 (Q)$. 
From there, we define:  
$  
H^1( (0,T) , L^2 (\Omega_t) ) := 
\left\lbrace
u \in L^2 ( (0,T) , L^2( \Omega_t) ) \mid \partial_t  u \in L^2 ( (0,T) , L^2( \Omega_t ) )
\right\rbrace 
$.

We finally introduce the Banach spaces $B(t)$ involved in the present study. 
We recall that 
$\Xq ( \Omega_t )^3$ and $\Yq ( \Omega_t )^3$ denote the closure of
$\lbrace \uu \in C^\infty (\Omega_t)^3 \mid \div \uu = 0 \rbrace$ in $\lV\cdot\rV_{W^{1,p}}$ and $\lV\cdot\rV_{L^{p}}$ norms respectively, 
and
$\Vq ( \Omega_t )^3$ and $\Hq ( \Omega_t )^3$ are the closure of
$\lbrace \uu \in C^\infty_0 (\Omega_t)^3 \mid \div \uu = 0 \rbrace$ in $\lV\cdot\rV_{W^{1,p}}$ and $\lV\cdot\rV_{L^{p}}$ norms respectively.
Then we define:
\begin{equation}
\begin{aligned}
\espX_q (\had) &:=
\left\lbrace 
\uu \in H^1 ( (0,T) , L^2  (\had)^3 ) \cap L^q ( (0,T) , \Vq (\had)^3 ) \ \big| \ \uu ( T, \cdot ) = 0
\right\rbrace  , \\
\espX_q^\pmS & :=
\left\lbrace 
\uu \in H^1 ( (0,T) , L^2 ( \Omega^\pmS_t )^3 ) \cap L^q ( (0,T) , \Vq ( \Omega^\pmS_t)^3 ) \ \big| \ \uu ( T, \cdot ) = 0
\right\rbrace  ,  
\end{aligned}
\label{eq:def:space:XqD}
\end{equation}
and we shall simply write $\espX_q$ when $\Omega^\pmS_t$ is replaced by $\Omega_t$. 
The condition $\uu ( T, \cdot ) = 0$ makes sense because it can be shown that the time evaluation is well-defined for any $\uu \in H^1( (0,T) , L^2 ( \Omega_t )^3 )$, and for any $t \in [0,T]$ holds $\uu (t) \in L^2_{\mrm{loc}} ( \Omega_t)^3$.

%*********************************************************************
%*********************************************************************
\subsection{Blood flows equations in moving domains}
\label{sec:blood:eq}

Let $((\Omega^\pmS , \vel^\pmS))_{\pmS\geq 0}$ be a sequence in $\Opn \times \V$.
We consider the following equations for blood flows posed on the resulting sequence of non cylindrical domains defined in \eqref{eq:def:Omegat-n}:
\begin{equation}
\label{pbm:blood_flow:hetero-Dirichlet:moving-dom}
\left\lbrace
\begin{aligned}
\frac{\partial \uun }{\partial t} + \uun \cdot \nabla \uun - \div ( \stress ( \Gs \uun ) ) + \gd \prn  &= \bd{f}   & \text{ in } Q^\pmS , \\
\div \uun &= 0   & \text{ in } Q^\pmS , \\
\uun(0) &= \uu_0^\pmS   & \text{ in }  \Omega^\pmS , \\ 
\uun &= \vel^\pmS   & \text{ on }  \Sigma^\pmS ,
\end{aligned}
\right.
\end{equation}
where  for any $\vv$, $ \Gs  \vv $ is defined in \eqref{eq:def:stress} and:
\begin{equation}
\label{eq:def:stress:2}
\stress ( \Gs \vv )  := \left( 1 + \lvert \Gs \vv \rvert \right)^{q-2}  \Gs \vv , 
\quad 
\text{for } 6/5 < q < 2. 
\end{equation}
We assume that for all $\pmS \in \NN$:

\smallskip

\begin{enumerate}[label=(U.\arabic*)%, leftmargin=8ex, itemsep=1.5ex
]
\item \label{item:u0:cond:1}
$\uun_0 \in L^2_{0,\div} (D)^3$,

\item \label{item:u0:cond:2}
$\uun_0 = \vel^\pmS (0)$ in $\had \setminus \Omega^\pmS$,

\item \label{item:u0:cond:3}
and $\uun_0 \rightharpoonup \uu_0$ weakly in $L^2 ( \had )$.
\end{enumerate}

\smallskip

With these assumptions, if we consider that $(\Omega^\pmS , \vel^\pmS ) \longrightarrow ( \Omega , \vel )$, then we have   $\uu_0 = \vel (0)$ in $\had \setminus \Omega$. 
From \cite[Thm. 49]{nagele_ruzicka2018}, there exists a \emph{weak solution} $ \uun $ to Problem~\eqref{pbm:blood_flow:hetero-Dirichlet:moving-dom}-\eqref{eq:def:stress:2}, the definition of which is given below.
\begin{definition}
\label{thm:def:wsol:blood:mov-dom}
Let the rheology exponent be $6/5 < q < 11/5$,  the body force 
$\bd{f} \in  L^{q'} ( \Qhad )^3 \inj L^{q'} ( Q^\pmS )^3$, and the initial data
$\uu_0^\pmS \in  L^2_{\div} ( \Omega^\pmS )^3  $ such that $\uu_0^\pmS - \vel^\pmS (0) \in L^2_{0,\div} ( \Omega^\pmS )^3 $.
We say that $ \uun \in L^\infty L^2_{\div} ( \Omega^\pmS_t )^3 \cap L^q \Xq ( \Omega^\pmS_t)^3$ is a weak solution of \eqref{pbm:blood_flow:hetero-Dirichlet:moving-dom}-\eqref{eq:def:stress:2}, if the following hods:
\begin{multline}
\label{pbm:blood:weak:moving-dom:1}
- \int_0^T ( \uun , \partial_t \bd{\eta} )_{L^2 ( \Omega^\pmS_t )}
+ \int_0^T \langle \Bg ( \uun )  , \bd{\eta} \rangle_{W^{1,p}( \Omega^\pmS_t )}
+ \int_0^T \langle \BT ( \uun ) , \bd{\eta} \rangle_{W^{1,p}( \Omega^\pmS_t )}  \\
= \int_0^T \langle \bd{f} , \bd{\eta} \rangle_{W^{1,p}( \Omega^\pmS_t )} + ( \uu_0^\pmS , \bd{\eta}(0) )_{L^2 (  \Omega^\pmS )} , 
\quad \forall \bd{\eta} \in \espX_p^\pmS ,
\end{multline}
for $p \geq   (5q/6)' $,  
where the function spaces are given in Section~\ref{sec:func:setting}, and where we define:
\begin{align} 
\langle \Bg ( \uu )  , \bd{\eta} \rangle_{W^{1,p}( \Omega^\pmS_t )} &:= - \int_{\Omega^\pmS_t} (  \uu (t) \otimes  \uu (t) ) : \gd \bd{\eta} (t) , \notag \\
\label{eq:def:Op:T}
\langle \BT ( \uu )  , \bd{\eta} \rangle_{W^{1,p}( \Omega^\pmS_t )} &:=   \int_{\Omega^\pmS_t} \stress ( \Gs \uu (t) ) : \Gs \bd{\eta} (t) .
\end{align} 
\end{definition}

{\it Let us comment on this definition.}
On the one hand, from the assumption on the regularity of $\uun$, Sobolev injection theorem, and interpolation, we have $\uun \otimes \uun \in L^{5q/6} L^{5q/6} (t)^{3\times 3}$, which guarantees the definition of $\Bg ( \uu ) $. 
Given the range of $q$, we have $p \geq 2$ allowing the $L^2$ products,  and we also have $q' \geq p'$, so that the term depending on the body force $\bd{f}$ is well-defined. 
On the other hand, 
it can be shown that for any $1 < q < \infty$, the stress tensor defined in \eqref{eq:def:stress:2} satisfies the following properties (see e.g., \cite[Chap.III, Sec. 2.2.1]{galdi_et-al2008hemodynamical}), 
for all $\bd{A},\bd{B} \in \RR^{\dn  \times \dn}_{\sym}$:

\smallskip

\begin{enumerate}[label=(\roman*)%, itemsep=1.5ex
]
\item Coercivity: \label{eq:def:coercivity}  
$
  \stress  ( \bd{A} ) : \bd{A}   \geq c_1 \lv \bd{A} \rv^q - c_2 ,
$
\item \label{eq:def:growth} 
Growth: 
$ 
\lv  \stress  ( \bd{A} ) \rv  \leq c_3 ( 1+ \lv  \bd{A}  \rv^{q-1} )
$,  
\item Monotonicity: \label{eq:def:monotonicity} 
$ 
(  \stress (\bd{A}) -  \stress (\bd{B}) ) : (\bd{A} - \bd{B}) > 0 ,
$  
\end{enumerate}
for some constants $c_1 , c_3 >0$ and $c_2 \in \RR$ depending on $q$.  
The regularity assumed for the weak solution and  the growth condition \ref{eq:def:growth} yield  
$\stress ( \Gs \uu ) \in L^{q'}L^{q'} (t)^{3\times 3}$, so that the term $\BT ( \uu )$ is well-defined.  
Classically, property \ref{eq:def:coercivity} is used to obtain the energy inequality in Proposition \ref{thm:estim:constant:unif},  and \ref{eq:def:monotonicity} allows for identifying the stress limit in Section \ref{sec:identif-stress}. 
Thus, more general rheological laws satisfying \ref{eq:def:coercivity}-\ref{eq:def:monotonicity} can suit the present study (see \cite{nagele_ruzicka2018}).

 %#####################################################################
%#####################################################################
\section{Shape continuity of blood flows}
\label{sec:continuity:blood}

In this section, we show that   
for a converging  sequence of moving domains $((\Omega^\pmS , \vel^\pmS ))_{\pmS \geq 0}$, and for an associated sequence of solutions $(\uu^\pmS  )_{\pmS \geq 0}$, one can extract a subsequence converging to the blood solution associated to the limit moving domain. 
This composes the following main result of this section.

\begin{theorem}
\label{thm:shapo-cont:blood}
Let $(\Omega^\pmS , \vel^\pmS ) \rightarrow ( \Omega , \vel )$ in $\Opn \times \V$,  
$\uun_0$  
satisfying \ref{item:u0:cond:1}-\ref{item:u0:cond:3}, 
and let $\uun$ be a weak solution of \eqref{pbm:blood_flow:hetero-Dirichlet:moving-dom}-\eqref{eq:def:stress:2}. 
Let $\extun$ be  
the extension of $\uun$ to $\Qhad$ by $\vel^\pmS$.  
Then, there exists $ \extu \in L^\infty L^2_{\div} ( \had )^3 \cap L^q \Xq ( \had )^3$ and a subsequence  
such that $\uu^{\pmS'}\rightharpoonup \extu$  weakly in $L^{ q } \Vq (  \had  )^3 $. 
\end{theorem}

In Section \ref{sec:exist:unif}, we first recall the existence result proved in \cite{nagele_ruzicka2018}, which gives solutions $\uun$ for each moving domain $(\Omega^\pmS , \vel^\pmS )$. 
Then we show, for these solutions, an energy estimate uniform with respect to the sequence of moving domains. 
From this, we extract in Section \ref{sec:exist:unif:limit} a converging subsequence as $\pmS \rightarrow \infty$, and identify the convective and stress limits. 
For doing this, we need to consider compactly supported test functions, so that we finally conclude with a density argument.

%*********************************************************************
%*********************************************************************
\subsection{Existence of solutions and uniform energy estimate}
\label{sec:exist:unif}

\begin{proposition}
\label{thm:existNestim}
Let 
$( \Omega^\pmS , \vel^\pmS ) \in \Opn \times \V$, 
$\bd{f} \in  L^{q'} ( \Qhad )^3$,  
and $\uun_0$ satisfy \ref{item:u0:cond:1}-\ref{item:u0:cond:3}. 
\begin{enumerate}%[itemsep=1.5ex]
\item \label{item:thm:existNestim:exist} 
Then there exists a weak solution $\uun  \in L^\infty L^2 ( \Omega^\pmS_t )^3 \cap L^p \Xq ( \Omega^\pmS_t )^3$ to Problem~\eqref{pbm:blood_flow:hetero-Dirichlet:moving-dom}-\eqref{eq:def:stress:2} written with initial data $\uun_0$.

\item \label{item:thm:existNestim:estim} 
Furthermore, the following estimate holds:
\begin{align*}
\lV  \uun \rV_{L^\infty L^2  ( \Omega^\pmS_t)}^2 +
\lV  \uun \rV_{L^q W^{1,q} ( \Omega^\pmS_t )}^q 
\leq C ,
\end{align*}
where $C = C(  \uu_0 , f_\vel , c_\V , c_{ \korn } , T , \lv \Qhad \rv )$, with $c_\V$ the bound for the norm of $\vel$ given in the definition \eqref{eq:def:class:V} of $\V$, and $ c_{ \korn } $ the Korn's constant from \eqref{eq:unif:Bogv-and-Korn}. 
\end{enumerate}

\end{proposition}

Part \ref{item:thm:existNestim:exist} of Proposition \ref{thm:existNestim} is shown in \cite{nagele_ruzicka2018}. 
We recall the procedure in Sections \ref{sec:homog-dirichlet} and \ref{sec:blood-approx},  
to have the elements needed to  show the uniform estimate of part \ref{item:thm:existNestim:estim} in Section \ref{sec:unif-estim}.

%~~~~~~~~~~~~~~~~~~~~~~~~~~~~~~~~~~~~~~~~~~~~~~~~~~~~~~~~~~~~~~~~~~~~~
%~~~~~~~~~~~~~~~~~~~~~~~~~~~~~~~~~~~~~~~~~~~~~~~~~~~~~~~~~~~~~~~~~~~~~
\subsubsection{Equivalent homogeneous Dirichlet problem}
\label{sec:homog-dirichlet}

To solve Problem~\eqref{pbm:blood_flow:hetero-Dirichlet:moving-dom}-\eqref{eq:def:stress:2}, 
the first step is to homogenize the boundary condition, that is to find: 
$ 
\ww^\pmS = \uun - \vel^\pmS  
$  
solution in $L^\infty L^2_{0,\div} ( \Omega^\pmS_t )^3 \cap L^q \Vq ( \Omega^\pmS_t )^3$ of problem \eqref{pbm:blood:weak:moving-dom:1} for which $\Bg ( \uu )$, $\BT ( \uu )$, and $\bd{f}$ are respectively replaced by:  
\begin{equation}
\label{eq:def:Op:T-V}
\Bg_{\vel^\pmS} ( \uu ) := \Bg ( \uu + \vel^\pmS ) , \quad
\BT_{\vel^\pmS} ( \uu ) := \BT ( \uu + \vel^\pmS) , \quad
 \text{ and } \quad
\bd{f}_{\vel^\pmS} := \bd{f} - \partial_t \vel^\pmS ,
\end{equation}
and where the term depending on the initial condition is $(  \ww_0^\pmS , \bd{\eta} (0) )_{L^2( \Omega^\pmS )}$, 
where $\ww_0^\pmS := \uu_0^\pmS - \vel^\pmS (0)$. 
Namely, 
we seek $\ww^\pmS$ as a solution of the following problem: 
\begin{multline}
\label{pbm:blood:weak:moving-dom:wn}
- \int_0^T ( \ww^\pmS , \partial_t \bd{\eta} )_{ L^2( \Omega^\pmS_t )}
+ \int_0^T \langle \Bg_{\vel^\pmS} ( \ww^\pmS )  , \bd{\eta} \rangle_{\Vp( \Omega^\pmS_t )}
+ \int_0^T  \langle \BT_{\vel^\pmS} ( \ww^\pmS ) , \bd{\eta} \rangle_{\Vp( \Omega^\pmS_t )}  \\
= \int_0^T \langle \bd{f}_{\vel^\pmS} , \bd{\eta} \rangle_{\Vp( \Omega^\pmS_t )} + ( \ww_0^\pmS , \bd{\eta}(0) )_{L^2( \Omega^\pmS )} , 
\quad \forall \bd{\eta} \in \espX_p^\pmS .
\end{multline}

%~~~~~~~~~~~~~~~~~~~~~~~~~~~~~~~~~~~~~~~~~~~~~~~~~~~~~~~~~~~~~~~~~~~~~
%~~~~~~~~~~~~~~~~~~~~~~~~~~~~~~~~~~~~~~~~~~~~~~~~~~~~~~~~~~~~~~~~~~~~~
\subsubsection{Approximation}
\label{sec:blood-approx}

In order to solve \eqref{pbm:blood:weak:moving-dom:wn},  the stress tensor $\stress ( \Gs \uu)$ is approximated by   
$ 
\stress^\pmR ( \Gs \uu) := \stress ( \Gs \uu) + \frac{1}{\pmR} \BA (\Gs \uu)  
$,  
where $\BA$ is given by: 
$ 
\label{eq:def:stress-p}
\BA ( \Gs \uu )  := \left( 1 + \lvert \Gs \uu \rvert \right)^{p-2}  \Gs \uu $,  
for $p$ given in Definition \ref{thm:def:wsol:blood:mov-dom}. 
We denote by $\BT^\pmR$ and $\BT^\pmR_{\vel^\pmS}$ the operators defined in \eqref{eq:def:Op:T} and \eqref{eq:def:Op:T-V} by replacing $\stress$ by $\stress^\pmR$, 
and denote by $\bB^\pmR$ and $\bB^\pmR_{\vel^\pmS}$ the operators defined in \eqref{eq:def:Op:T} and \eqref{eq:def:Op:T-V} by replacing $\stress$ by $(1/\pmR) \BA$. 
Finally the initial data is regularized: we consider  
$(\ww_{0,\pmR}^{\pmS})_{\pmR\geq 0} \subset \Vp ( \Omega^{\pmS} )^3$, such that $\ww_{0,\pmR}^{\pmS} \longrightarrow \ww_0^{\pmS}$ strongly in $L^2$ as $\pmR \rightarrow \infty$. 
We notice that 
$\bd{f} \in  L^{q'} ( \Qhad )^3 \inj L^{q'} ( Q^\pmS )^3 \inj L^{p'} ( Q^\pmS )^3 $.
Thus we can apply the following result (see \cite[Prop. 52 and Lem. 56]{nagele_ruzicka2018}).

\begin{proposition}
\label{thm:exist:approx:blood}
Let the body force  
$\bd{f} \in  L^{q'} ( \Qhad )^3$, 
and let $( \Omega^\pmS , \vel^\pmS ) \in \Opn \times \V$. 
\begin{enumerate}%[itemsep=1.5ex]
\item \label{item:thm:exist:approx:blood:i}
Let the initial data be
$\ww_{0,\pmR}^{\pmS} \in \Vp ( \Omega^\pmS )^3$.
There exists a  
$\ww_\pmR^\pmS  \in L^\infty L^2_{0,\div} ( \Omega^\pmS_t )^3 \cap L^p \Vp ( \Omega^\pmS_t)^3$ weak solution to:
\begin{multline}
\label{pbm:blood:weak:moving-dom:wnN}
\int_0^T 
\Big[
- ( \ww^\pmS_\pmR , \partial_t \bd{\eta} )_{L^2 (\Omega^\pmS_t)}
+ 
\langle \Bg_{\vel^\pmS} ( \ww^\pmS_\pmR )  , \bd{\eta} \rangle_{W^{1,p}(\Omega^\pmS_t)}
+ 
\langle \BT_{\vel^\pmS}^\pmR ( \ww^\pmS_\pmR ) , \bd{\eta} \rangle_{W^{1,p}(\Omega^\pmS_t)} 
\Big]
\\
= \int_0^T \langle \bd{f}_{\vel^\pmS} , \bd{\eta} \rangle_{W^{1,p}(\Omega^\pmS_t)} + ( \ww_{0,\pmR}^{\pmS} , \bd{\eta}(0) )_{L^2 (\Omega^\pmS)} , 
\quad \forall \bd{\eta} \in \espX_p^\pmS .
\end{multline}

\item \label{item:thm:exist:approx:blood:ii}
We have that for every $\pmR \in \NN$,  every weak solution $\ww^\pmS_\pmR$, every $s \in [0,T]$, and every $\bd{\eta} \in H^{1} L^2 (\Omega^\pmS_t)^3 \cap L^p \Vp (\Omega^\pmS_t)^3$, 
it holds:
\begin{multline}
\label{pbm:blood:weak:moving-dom:2}
( \ww^\pmS_{\pmR}(s) , \bd{\eta}(s) )_{L^2 (\Omega^\pmS_s)} 
- ( \ww^\pmS_{0,\pmR} , \bd{\eta}(0) )_{L^2 (\Omega^\pmS)} 
+ \int_0^s \langle \BT^\pmR_{\vel^\pmS} ( \ww^\pmS_\pmR ) , \bd{\eta} \rangle_{W^{1,p}(\Omega^\pmS_t)} \\
=
\int_0^s \langle \bd{f}_{\vel^\pmS} , \bd{\eta} \rangle_{W^{1,p}(\Omega^\pmS_t)} 
- \int_0^s \langle \Bg_{\vel^\pmS} ( \ww^\pmS_\pmR )  , \bd{\eta} \rangle_{W^{1,p}(\Omega^\pmS_t)} 
+ \int_0^s ( \ww^\pmS_\pmR ,  \partial_t \bd{\eta} )_{L^2 (\Omega^\pmS)}  .
\end{multline}

\item \label{item:thm:exist:approx:blood:iii}
There exists a not relabeled subsequence $(\ww_\pmR^\pmS)_{\pmR\geq 0} $
such that:
\begin{align*}
\ww^\pmS_\pmR  & \xrightharpoonup[\pmR\rightarrow \infty]{}  \ww^\pmS
\qquad   \text{ weakly in } L^{ q } \Vq (  \Omega^\pmS_t  )^3 , \\
\ww^\pmS_\pmR   & \xrightharpoonup[\pmR\rightarrow \infty]{}  \ww^\pmS
\qquad   \text{ weakly-$*$ in } L^{ \infty } L^2 (  \Omega^\pmS_t  )^3 ,
\end{align*}
where $\ww^\pmS\in L^\infty L^2_{0,\div} ( \Omega^\pmS_t )^3 \cap L^q \Vq ( \Omega^\pmS_t)^3$ is a weak solution to Problem \eqref{pbm:blood:weak:moving-dom:wn}. 
\end{enumerate}

\end{proposition}

The proof of Proposition \ref{thm:exist:approx:blood} is given in \cite[Sec. 5]{nagele_ruzicka2018}. 
To obtain this result, the authors write an energy estimate for the solutions $\ww^\pmS_\pmR$, allowing for letting  $\pmR \rightarrow \infty$. 
In order to pass to the limit as $\pmS \rightarrow \infty$, we show that this energy estimate is uniform.

%~~~~~~~~~~~~~~~~~~~~~~~~~~~~~~~~~~~~~~~~~~~~~~~~~~~~~~~~~~~~~~~~~~~~~
%~~~~~~~~~~~~~~~~~~~~~~~~~~~~~~~~~~~~~~~~~~~~~~~~~~~~~~~~~~~~~~~~~~~~~
\subsubsection{Uniform energy estimate}
\label{sec:unif-estim}

\begin{proposition}
\label{thm:estim:constant:unif}
Let $\ww_\pmR^\pmS$ be a  
solution of Problem \eqref{pbm:blood:weak:moving-dom:wnN}.
Then, for some constant $C>0$ which does not depend on $\pmS$, the following estimate holds:
\begin{align}
\label{eq:estim:blood:1}
\lV  \ww_\pmR^\pmS \rV_{L^\infty L^2(\Omega^\pmS_t)}^2 +
\lV  \ww_\pmR^\pmS \rV_{L^q W^{1,q} (\Omega^\pmS_t)}^q +
\frac{1}{\pmR} \lV \Gs  \ww_\pmR^\pmS \rV_{L^p ( Q^\pmS )}^p
\leq C .
\end{align}
\end{proposition}

\begin{proof}
We recall that $( \ww^\pmS_{0,\pmR} )_{\pmR\geq 0}$ is chosen in $ \Vp ( \Omega^{\pmS} )^3$ such that $\ww^\pmS_{0,\pmR}$ converges strongly in $L^2 (\Omega^\pmS)^3$ to $\ww^\pmS_{0}\in L^2_{0,\div} ( \Omega^\pmS )^3$, 
and the solution $\ww^\pmS_{\pmR}$ is such that $\ww^\pmS_{\pmR}(0) =  \ww^\pmS_{0,\pmR}$. 
Thus for $\pmR$ large enough, 
$( \ww^\pmS_{0,\pmR} , \ww^\pmS_{\pmR} (0) )_{L^2 (\Omega^\pmS)} = \lV \ww^\pmS_{0,\pmR} \rV^2_{L^2 ( \Omega^\pmS )} 
\leq 2 \lV \ww^\pmS_{0} \rV^2_{L^2 ( \Omega^\pmS )}$. 
From the condition \ref{item:u0:cond:2} on $\uun_0$ and the definition of $\ww_0^\pmS$, we have that 
$\lV \ww_0^\pmS \rV_{L^2 ( \Omega^\pmS )} 
\leq \lV \uun_0 \rV_{L^2 ( \had )} 
+\lV \vel^\pmS(0) \rV_{L^2 ( \had )}$. 
Furthermore, from the weak convergence of $( \uun_0 )_{\pmS\geq 0}$ in $L^2 ( \had )^3$, we have a positive constant $C_0$ such that $ \lV \uun_0 \rV_{L^2 ( \had )} \leq C_0$ for all $\pmS\in \NN$,  
and by definition of $(\Omega^\pmS , \vel^\pmS ) \longrightarrow ( \Omega , \vel )$, we have $\lV \vel^\pmS \rV_{L^2 ( \had )} \leq \lv \had \rv^{1/2} c_\V$ for all $\pmS\in \NN$.
Due to these considerations, by setting $\bd{\eta} = \ww^\pmS_\pmR$ in \eqref{pbm:blood:weak:moving-dom:2}, and rewriting the last term of the right hand side, we get:
\begin{align}
\frac{1}{2} ( \ww^\pmS_{\pmR}(s) , \ww^\pmS_\pmR (s) )_{L^2( \Omega^\pmS_s)} 
&
+ \int_0^s ( \BT^\pmR_{\vel^\pmS} \langle \ww^\pmS_\pmR ) , \ww^\pmS_\pmR \rangle_{W^{1,p}(\Omega^\pmS_t)} \notag\\
&\leq 
\int_0^s \langle \bd{f}_{\vel^\pmS} ,\ww^\pmS_\pmR  \rangle_{W^{1,p}(\Omega^\pmS_t)}
- \int_0^s \langle \Bg_{\vel^\pmS} ( \ww^\pmS_{\pmR} )   , \ww^\pmS_\pmR \rangle_{W^{1,p}(\Omega^\pmS_t)}
+  K_0  ,
\label{pbm:blood:weak:moving-dom:3}
\end{align}
where $K_0 = 2 (C_0^2 + \lv \had \rv c_\V^2 )$. We can directly compute that:
\begin{equation}
\label{pbm:blood:weak:moving-dom:4}
- \int_0^s \langle \Bg_{\vel^\pmS} ( \ww^\pmS_{\pmR} )   , \ww^\pmS_\pmR \rangle_{W^{1,p}(\Omega^\pmS_t)}
= \int_0^s \int_{\Omega^\pmS_t} ( \vel^\pmS \otimes \vel^\pmS ) : \gd \ww^\pmS_\pmR 
- \int_0^s \int_{\Omega^\pmS_t} ( \ww^\pmS_\pmR \otimes \ww^\pmS_\pmR ) : \gd  \vel^\pmS.
\end{equation}
The second term of the right hand side of \eqref{pbm:blood:weak:moving-dom:4} is bounded by $c_\V \int_0^s \lV \ww^\pmS_\pmR \rV_{L^2 ( \Omega^\pmS_s )}$. 
Regarding the first term, we would like to have the $L^q$-norm of $\gd \ww^\pmS_\pmR$. 
Recalling that $(\vel^\pmS \otimes \vel^\pmS ) \in L^\infty ( \Qhad )^{3\times 3}$, we can write: 
$$
\int_0^s \int_{\Omega_t} ( \vel^\pmS \otimes \vel^\pmS ) : \gd \ww^\pmS_\pmR
\leq \int_0^s 
\lV \vel^\pmS \otimes \vel^\pmS \rV_{L^{q'} ( \Omega^\pmS_t )} 
\lV \gd \ww^\pmS_\pmR \rV_{L^{q} ( \Omega^\pmS_t )} .
$$ 
Then, the first term of the right hand side of \eqref{pbm:blood:weak:moving-dom:4} can be bounded by:
\begin{align*}
\int_0^s \lV \vel^\pmS \otimes \vel^\pmS \rV_{L^{q'} ( \Omega^\pmS_t )} 
\lV \gd \ww^\pmS_\pmR  \rV_{L^q ( \Omega^\pmS_t )}  
& \leq \int_0^s \Big(
\frac{\varepsilon}{q}  \lV \Gs \ww^\pmS_\pmR  \rV_{L^q ( \Omega^\pmS_t )}^q 
+ \frac{ c_{\korn}^{q'} }{q' \varepsilon^{1/(q-1)}} \lV \vel^\pmS \otimes \vel^\pmS \rV_{L^{q'} ( \Omega^\pmS_t )}^{q'} 
\Big) \notag \\
&\leq
\frac{\varepsilon }{q} \int_0^s \lV \Gs \ww^\pmS_\pmR  \rV_{L^q ( \Omega^\pmS_t )}^q 
+
\frac{ c_{\korn}^{q'} c_{\V}^{2q'} \lv \Qhad \rv}{q' \varepsilon^{1/(q-1)} }     ,
\end{align*}
where $c_{\korn}$ is an upper bound for Korn's constant of $\Omega^\pmS_t$, given in \eqref{eq:unif:Bogv-and-Korn_t}. 
Now we consider the term depending on $\bd{f}_{\vel^\pmS}$ in \eqref{pbm:blood:weak:moving-dom:3}, 
rewritten $\int_0^s \langle \bd{f}_{\vel^\pmS} , \ww^\pmS_\pmR \rangle_{W^{1,p}(\Omega^\pmS_t)} = \int_0^s \langle \bd{f}_{\vel^\pmS} , \ww^\pmS_\pmR \rangle_{W^{1,q}(\Omega^\pmS_t)} $. 
Noting that Poincar\'{e}'s inequality holds uniformly with constant $c_{\nP}$ (see \cite[Prop. 3.1.17]{henrot_pierre2018shapo_geom}), we estimate: 
\begin{align*}
\int_0^s \langle \bd{f}_{\vel^\pmS} , \ww^\pmS_\pmR \rangle_{W^{1,q}(\Omega^\pmS_t)}
&\leq c_{\nP}c_{\korn}
\int_0^s \lV \bd{f}_{\vel^\pmS} \rV_{L^{q'} ( \Omega^\pmS_t )}
\lV \Gs \ww^\pmS_\pmR \rV_{L^q ( \Omega^\pmS_t )} \notag \\
&\leq 
\frac{\varepsilon }{q} \int_0^s \lV \Gs \ww^\pmS_\pmR  \rV_{L^q ( \Omega^\pmS_t )}^q 
+ 
\frac{(c_{\nP}c_{\korn})^{q'}  }{ q' \varepsilon^{1/(q-1)} }  \int_0^s \lV \bd{f}_{\vel^\pmS} \rV_{L^{q'} ( \Omega^\pmS_t )}^{q'} \notag \\
&\leq 
\frac{\varepsilon }{q} \int_0^s \lV \Gs \ww^\pmS_\pmR  \rV_{L^q ( \Omega^\pmS_t )}^q 
+ 
\frac{ (c_{\nP}c_{\korn})^{q'} 2^{1/(q-1)} }{ q' \varepsilon^{1/(q-1)} } \Big( \lV \bd{f} \rV_{L^{q'} ( \Qhad )}^{q'}
+  c_{\V}^{q'}  \lv \Qhad \rv 
\Big) .
\end{align*}
It can be shown (see e.g. \cite[Lem. 25]{nagele_ruzicka2018}) that $\langle \BT_{\vel^\pmS} ( \ww^\pmS_\pmR ) , \ww^\pmS_\pmR \rangle_{W^{1,p}(\Omega^\pmS_t)} \geq \lV \Gs \ww^\pmS_\pmR \rV_{L^q( \Omega^\pmS_t )}^q $ and $\langle \bB^\pmR_{\vel^\pmS} ( \ww^\pmS_\pmR ) , \ww^\pmS_\pmR \rangle_{W^{1,p}(\Omega^\pmS_t)} \geq \frac{1}{\pmR} \lV \Gs \ww^\pmS_\pmR \rV_{W^{1,p}(\Omega^\pmS_t)}^p $, where we recall that $\BT^\pmR_{\vel^\pmS} = \BT_{\vel^\pmS} + \bB^\pmR_{\vel^\pmS}$. Gathering all these results, and choosing $\varepsilon$ small enough, we get: 
\begin{multline*}
 \frac{1}{2} \left[ 
( \ww^\pmS_{\pmR}(s) , \ww^\pmS_\pmR (s) )_{L^2(\Omega^\pmS_s)} 
+ \int_0^s \Big(
\lV \Gs \ww^\pmS_\pmR \rV_{L^q( \Omega^\pmS_t)}^q 
+ \frac{1}{\pmR}  \lV \Gs \ww^\pmS_\pmR \rV_{W^{1,p}(\Omega^\pmS_t)}^p
\Big) \dm t
\right]
  \\
\leq  K_0 
 + \frac{ (c_{\nP}c_{\korn})^{q'}  }{ q' \varepsilon^{1/(q-1)} }
\left(
2^{1/(q-1)} \lV \bd{f} \rV_{L^{q'} ( \Qhad )}^{q'}
+ c_{\V}^{q'} ( 2^{1/(q-1)} + c_{\V}^{q'} ) \lv \Qhad \rv
\right) 
+ c_\V \int_0^s \lV \ww^\pmS_\pmR \rV_{L^2 ( \Omega^\pmS_s )} . 
\end{multline*}
Thus we obtain from Gr\"{o}nwall's Lemma:
\begin{equation}
\lV  \ww_\pmR^\pmS \rV_{L^\infty L^2  (\Omega^\pmS_t)}^2 +
\lV  \ww_\pmR^\pmS \rV_{L^q W^{1,q}(\Omega^\pmS_t)}^q   
+
\frac{1}{\pmR} \lV \Gs  \ww_\pmR^\pmS \rV_{L^p ( Q^\pmS )}^p
\leq  
C , 
\label{eq:estim:blood:2} 
\end{equation}
for some $C =C ( q , C_0 , \lv \Qhad \rv , T , \lV\bd{f}\rV_{L^{q'}(\Qhad)} , c_\V , c_\korn )$.  
This completes the proof.
\end{proof}

%*********************************************************************
%*********************************************************************
\subsection{Continuity of blood flows solutions}
\label{sec:exist:unif:limit}

Now we can pass to the limit as $\pmS\rightarrow \infty$. 
First, in view of Proposition \ref{thm:exist:approx:blood}, we recall that $\ww^\pmS$ is both the weak limit in $L^q \Vq ( \Omega^\pmS_t )^3$ and the weak-$*$ limit in $L^\infty L^2(\Omega^\pmS_t)^3$ of a sequence of solutions $(\ww_\pmR^\pmS)_{\pmR\geq 0}$, 
and we can estimate:  
\begin{align*}
\lV  \ww^\pmS \rV_{L^\infty L^2 ( \Omega^\pmS_t )} 
\leq \liminf_{\pmR\rightarrow \infty} \lV  \ww_\pmR^\pmS \rV_{L^\infty L^2 ( \Omega^\pmS_t )}   
\quad \text{ and } \quad
\lV  \ww^\pmS \rV_{L^q W^{1,q}( \Omega^\pmS_t )}^q 
\leq \liminf_{\pmR\rightarrow \infty} \lV  \ww_\pmR^\pmS \rV_{L^q W^{1,q}( \Omega^\pmS_t )}^q  .
\end{align*} 
Thus from \eqref{eq:estim:blood:1} and Proposition \ref{thm:estim:constant:unif} we have:
\begin{align}
\label{eq:estim:blood:wn}
\lV  \ww^\pmS \rV_{L^\infty L^2 ( \Omega^\pmS_t )}^2 +
\lV  \ww^\pmS \rV_{L^q W^{1,q} ( \Omega^\pmS_t )}^q 
\leq C .
\end{align}
We define on $\Qhad$ the zero  extension of $\ww^\pmS$ by:
\begin{equation*}
\label{eq:def:tilde_wn}
\widetilde{\ww}^\pmS := 
\left\lbrace
\begin{aligned}
&\ww^\pmS && \text{ in } Q^\pmS , \\
& 0 && \text{ in } \Qhad\setminus\overline{Q^\pmS} , \\
\end{aligned}
\right.
\end{equation*}
and we obviously have $\lV  \widetilde{\ww}^\pmS \rV_{L^\infty L^2 ( \had )} = \lV  \ww^\pmS \rV_{L^\infty L^2 ( \Omega^\pmS_t )}$ and $ \lV  \widetilde{\ww}^\pmS \rV_{L^q W^{1,q} ( \had )} = \lV  \ww^\pmS \rV_{L^q W^{1,q}( \Omega^\pmS_t )}$. 
We remark that $\widetilde{\ww}^\pmS$ is well-defined and belongs to $L^q \Vq ( \had )^3$. 
Indeed, because $\ww^\pmS $ belongs to $L^q \Vq ( \Omega^\pmS_t )^3$, 
we have by definition that $\ww^\pmS (t) \in \Vq ( \Omega^\pmS_t )^3$ for a.e. $t \in (0,T)$, and then the zero extension $\widetilde{\ww}^\pmS (t)$ belongs to $\Vq ( \had )^3$ for a.e. $t$ as well, and  $\lV \widetilde{\ww}^\pmS (t) \rV_{\Vq ( \had )} = \lV \ww^\pmS (t) \rV_{\Vq ( \Omega^\pmS_t )} $. 
The latter lying in $L^q ((0,T))^3$ by the definition of the generalized Bochner space, the classical Bochner regularity of $ \widetilde{\ww}^\pmS$ holds.    
Thus from \eqref{eq:estim:blood:wn}, Sobolev embedding, and growth condition \ref{eq:def:growth} on $\stress$, 
we can take a subsequence still denoted by $(\widetilde{\ww}^\pmS)_{\pmS\geq 0}$ such that:  
\begin{align}
\label{eq:weaklim:tild:w:1}
\widetilde{\ww}^\pmS  & \xrightharpoonup[\quad]{}  
\widetilde{\ww} 
&& \hspace{-1ex} \text{ weakly in } L^{ q } \Vq (  \had  )^3 \text{ and } L^{ 5q / 3 } (  \Qhad  )^3 , \\
\label{eq:weaklim:tild:w:2}
( \widetilde{\ww}^\pmS + \vel^\pmS ) \otimes ( \widetilde{\ww}^\pmS  + \vel^\pmS ) &\xrightharpoonup[\quad]{}  
\mathbf{H}_{\vel}
&& \hspace{-1ex} \text{ weakly in }  L^{ 5q / 6 } (  \Qhad  )^{3 \times 3} , \\
\label{eq:weaklim:tild:w:5}
\stress ( \Gs \widetilde{\ww}^\pmS + \Gs \vel^\pmS ) &\xrightharpoonup[\quad]{} 
\bd{\chi}_{\vel}
&& \hspace{-1ex} \text{ weakly in } L^{ q' }  (  \Qhad  )^{3 \times 3} ,
\end{align}
where $\widetilde{\ww}$, $\mathbf{H}_{\vel}$, and $\bd{\chi}_{\vel}$ denote the weak limits, 
and 
$\widetilde{\ww}^\pmS  \rightharpoonup
\widetilde{\ww}$ weakly-$*$ in $L^{ \infty } L^2 (  \had  )^3$.

\begin{proposition}
\label{thm:prop:extension:Xp:Q:2:Qhad}
Let us consider $(\Omega ,\vel )\in \Opn \times \V$, and let $\bd{\eta} \in \espX_q$ the corresponding function space be defined in \eqref{eq:def:space:XqD}.
Then $\bd{\widetilde{\eta}} \in \espX_q ( \had )$, where $\bd{\widetilde{\eta}}$ denotes the zero extension of $\bd{\eta}$ to $\Qhad \setminus Q$, and $\espX_q ( \had )$ is defined in \eqref{eq:def:space:XqD}.
\end{proposition}

\begin{remark} ~
\begin{itemize}
    \item It is clear that $\bd{\widetilde{\eta}} \in L^p \Vq ( D )^3$. 
\item We refer to \cite[Lem. 2.11]{nagele2015monotone}, ensuring that the zero extension of a function in $ H^1 L^2 ( \Omega_t )^3 \cap L^p W^{1,1}_0 ( \Omega_t )^3$ to $\RR^{\dn}$ lies in the space $H^1 L^2 (\RR^{\dn})^3 \inj C ( (0,T) , L^2 (\RR^{\dn})^3 )$.
\end{itemize}
\end{remark}

\smallskip

From this, we have $ \forall \bd{\eta}  \in \espX_p^\pmS$, still denoting by $\bd{\eta}$ its zero extension to $\Qhad$:  
\begin{align}
\label{eq:egalite:extension:1}
\langle \BT_{\vel^\pmS} ( \widetilde{\ww}^\pmS  ) , \bd{\eta} \rangle_{ W^{1,p}(\had) } 
&=   \langle \BT_{\vel^\pmS} ( \ww^\pmS  ) , \bd{\eta} \rangle_{W^{1,p}(\Omega^\pmS_t)} , 
&&& ( \widetilde{\ww}^\pmS  , \partial_t \bd{\eta} )_{ L^2 (\had) } 
&=  (  \ww^\pmS  , \partial_t \bd{\eta} )_{L^2( \Omega^\pmS_t )} , \\  
\langle \Bg_{\vel^\pmS} ( \widetilde{\ww}^\pmS )  , \bd{\eta} \rangle_{W^{1,p}(\had) } 
&=   \langle \Bg_{\vel^\pmS} ( \ww^\pmS )  , \bd{\eta} \rangle_{W^{1,p}(\Omega^\pmS_t)} , 
&&&  \langle \bd{f}_{\vel^\pmS} , \bd{\eta} \rangle_{W^{1,p}(\had) }  
&=   \langle \bd{f}_{\vel^\pmS} , \bd{\eta} \rangle_{W^{1,p}(\Omega^\pmS_t)} ,   \\  
\label{eq:egalite:extension:5}
(  \ww^\pmS_{0} , \bd{\eta}(0) )_{ L^2 (\had) } 
 &= ( \ww^\pmS_{0} , \bd{\eta}(0) )_{L^2 ( \Omega^\pmS )} ,  & &&& 
\end{align}
for a.e. $t \in (0,T)$, 
where we recall that $\ww^\pmS_{0} \in L^2 (\had)^3$ is given by $\ww^\pmS_{0} := \uun_0 - \vel^{\pmS} (0)$.  
Passing to the limit as $\pmR\rightarrow \infty$ in \eqref{pbm:blood:weak:moving-dom:wnN} we obtained \eqref{pbm:blood:weak:moving-dom:wn},
and then from \eqref{eq:egalite:extension:1}-\eqref{eq:egalite:extension:5} we have: 
\begin{multline}
\label{eq:wn:momentum}
- \int_0^T ( \widetilde{\ww}^\pmS  , \partial_t \bd{\eta} )_{ L^2 ( \had )} 
+ \int_0^T ( \BT_{\vel^\pmS} ( \widetilde{\ww}^\pmS  ) , \bd{\eta} )_{W^{1,p}( \had )}
+ \int_0^T ( \Bg_{\vel^\pmS} ( \widetilde{\ww}^\pmS )  , \bd{\eta} )_{W^{1,p}( \had )}  \\
=
\int_0^T ( \bd{f}_{\vel^\pmS} , \bd{\eta} )_{W^{1,p}( \had )} 
+ ( \ww^\pmS_{0} , \bd{\eta}(0) )_{L^2 ( \had )} , 
\quad
\forall \bd{\eta} \in \espX_p^\pmS .
\end{multline}

Now, let $\bd{\eta} \in \espX_p$ such that $\supp \bd{\eta} \subset \overline{ Q } \setminus \Sigma$. 
As we saw in Section~\ref{sec:shapo:setting} Lemma~\ref{thm:lem:compact-in-Q}, there exists a $\pmS_0 \in \NN$ such that for all $\pmS \geq \pmS_0$, $\bd{\eta} \in \espX_p^\pmS$. 
Thus, from the limits written in \eqref{eq:weaklim:tild:w:1}-\eqref{eq:weaklim:tild:w:5}, 
assumption \ref{item:u0:cond:3}, 
convergence of $(\vel^{\pmS})_{\pmS \geq 0}$,  
and \eqref{eq:wn:momentum}, we have that:
\begin{multline}
\label{eq:wn:momentum:lim:1}
- \int_0^T ( \widetilde{\ww}  , \partial_t \bd{\eta} )_{L^2 ( \Omega_t )} 
+ \int_0^T \langle \bd{\chi}_{\vel} , \bd{\eta} \rangle_{W^{1,p}( \Omega_t )}
+ \int_0^T \langle \mathbf{H}_{\vel}  , \bd{\eta} \rangle_{W^{1,p}( \Omega_t )}  \\
=
\int_0^T \langle \bd{f} - \partial_t \vel , \bd{\eta} \rangle_{W^{1,p}( \Omega_t )} 
+ ( \ww_{0} , \bd{\eta}(0) )_{L^2(\Omega)}  .
\end{multline}

%~~~~~~~~~~~~~~~~~~~~~~~~~~~~~~~~~~~~~~~~~~~~~~~~~~~~~~~~~~~~~~~~~~~~~
%~~~~~~~~~~~~~~~~~~~~~~~~~~~~~~~~~~~~~~~~~~~~~~~~~~~~~~~~~~~~~~~~~~~~~
\subsubsection{Identification of $\mathbf{H}_{\vel}$}

In order to identify nonlinear terms, classical compactness lemmas cannot be directly applied in the context of moving domains. 
Among other methods, the Landes–Mustonen compactness principle is used in \cite{nagele_ruzicka2018} to get strong convergence of the velocities in some Lebesgue space, and thus identify the convective term. 
Because of the sequence of moving domains, we cannot apply this method directly. 
Indeed, we need to write \eqref{eq:wn:momentum:lim:1} for test functions compactly supported in $Q$, and on such subsets, the velocities do not have zero-traces, so that the Leray–Helmholtz decomposition does not hold. 
There are other generalizations of compactness principles applied to moving domains (see, e.g. \cite{moussa2016aubin-lions, muha2019generalization}), but they don't apply to the present shape optimization framework either.  
To overcome this, we apply a local pressure representation method, based on \cite{wolf2017local-pressure}, 
that allows us to identify the limit of the convective term without strong convergence. 

We start with a preliminary step on results relating to local pressure representation. 
The following proposition combines results from
\cite[Lem. 6.3. \& Cor. 6.5.]{wolf2017local-pressure}.

\begin{proposition}
\label{thm:prop:local-pressure}
Let $J$ be a time interval, 
and let $G$ and $K$ be two bounded open sets of $\RR^{\dn}$ such that $G \Ksub K$ and $G$ has a $C^2$-boundary. 
Let $\ww \in L^q ( J , W_{\div}^{1,q} (K)^{\dn})$ and $\FF \in (L^p ( J , W_0^{1,p} (K)^{\dn})'$ satisfy the following:
\begin{equation*}
\int_{J \times K} \ww   \cdot \partial_t \bd{\psi}  \dm x \dm t 
= \int_J \langle \FF (t) ,  \bd{\psi}(t  ) \rangle_{W^{1,p}_0(K)} \dm t ,
\quad
\forall \bd{\psi} \in C^{\infty}_{0,\div_x} ( J \times K ) ,  
\end{equation*}
where $ C^{\infty}_{0,\div_x} ( J \times K ) := \lbrace \bd{\psi} \in C^{\infty}_0 ( J \times K ) ^3 \mid \div_x ( \bd{\psi} ) =0 \rbrace$.  
Then there exists local pressures $p_0 \in L^{p'} ( J \times G )$ and $p_h \in L^q ( J , W^{1,q} (G))$ such that the following holds: 
\begin{multline}
\label{eq:local-pressure:1}
\int_{J \times G} ( \ww + \nabla p_h )   \cdot \partial_t \bd{\psi}  \dm x \dm t 
= \int_J \langle \FF (t) ,  \bd{\psi}(t  ) \rangle_{W^{1,p}_0(G)} \dm t  \\
- \int_{J \times G} p_0 \div ( \bd{\psi} ) \dm x \dm t  ,
\quad
\forall \bd{\psi} \in C^{\infty}_{0} ( J \times G ) .
\end{multline}
Furthermore, there exists a constant $c(G)>0$ such that, for a.e. $t \in J$, we have the following  estimations 
 (first inequality holds if we also have $\ww (t) \in L^2 (G)$):  
\begin{equation}
\label{eq:local-pressure:2}
\begin{aligned}
\lV p_h (t) \rV_{L^2 (G)} 
& \leq c(G) \lV \ww (t) \rV_{L^2 (G)} , \\
\lV p_h (t) \rV_{W^{1,q} (G)} 
& \leq c(G) \lV \ww (t) \rV_{L^q (G)} , \\
\lV p_0 (t) \rV_{L^{p'} (G)} 
& \leq c(G) \lV \FF (t) \rV_{W_0^{1,p} (G)'} ,
\end{aligned}
\end{equation}
and $p_h$ is harmonic, so that we finally have for all $1\leq r \leq \infty$ (see, e.g. \cite{diening_ruzicka_wolf2010existence}): 
\begin{equation}
\label{eq:local-pressure:2bis}
\forall V \Ksub G , 
\quad 
\lV p_h (t) \rV_{W^{2,r} (V)} 
 \leq c(V) \lV  p_h (t)  \rV_{L^2 (G)} . 
\end{equation}
\end{proposition}
We complete the preliminary step 
with the following result, whose proof is inspired from 
\cite{bucur_feireisl_necasova_wolf2008asymp}. 
\begin{proposition}
\label{thm:prop:identif:convect}
Let $(\widetilde{\ww}^\pmS)_{\pmS \geq 0}$ be the sequence of functions satisfying \eqref{eq:wn:momentum} given in Section~\ref{sec:exist:unif:limit}, weakly converging in the sense of \eqref{eq:weaklim:tild:w:1} to $\widetilde{\ww}$ which satisfies \eqref{eq:wn:momentum:lim:1}. 
Let $J$ be a time interval and $G\Ksub \had$ be a $C^2$ open set such that $J \times \overline{G} \subset Q$. 
Then, recalling that $( \widetilde{\ww}^\pmS + \vel^\pmS ) \otimes ( \widetilde{\ww}^\pmS  + \vel^\pmS ) 
\rightharpoonup 
\mathbf{H}_{\vel}$ weakly in $L^{ 5q / 6 } (  \Qhad  )^{3 \times 3}$, we have: 
\begin{equation}
\label{eq:local-pressure:3}
\int_J \int_G \mathbf{H}_{\vel} : \nabla \bd{\psi} \dm x \dm t 
= 
\int_J \int_G ( \widetilde{\ww} + \vel ) \otimes ( \widetilde{\ww}   + \vel  )  : \nabla \bd{\psi} \dm x \dm t , 
\quad 
\forall \bd{\psi} \in C^{\infty}_{0,\div_x} ( J \times G ). 
\end{equation}
\end{proposition}

\begin{proof}
Let us consider some open set $K$ such that $G \Ksub K$ and $J\times \overline{K} \subset Q$. 
From Lemma \ref{thm:lem:compact-in-Q}, there exists $\pmS_0 \in \NN$ such that for all $\pmS \geq \pmS_0$, $J\times \overline{K} \subset Q^{\pmS}$. 
In view of what has been done in Section~\ref{sec:exist:unif:limit}, 
we have that for all $\pmS  \geq \pmS_0$, $\widetilde{\ww}^\pmS$ and $\FF^{\pmS}$ given by \eqref{eq:wn:momentum} satisfy the assumptions of Proposition~\ref{thm:prop:local-pressure}. 
Then, there exist local pressures $p_0^{\pmS}$ and $p_h^{\pmS}$ satisfying \eqref{eq:local-pressure:1} and \eqref{eq:local-pressure:2}, from which we deduce that $  ( \partial_t ( \widetilde{\ww}^\pmS + \nabla p_h^{\pmS} ) )_{\pmS \geq \pmS_0}$ is bounded in $(L^p ( J , W_0^{1,p} (G)^{\dn})'$. 
We recall that from \eqref{eq:estim:blood:wn}, we have that $(\widetilde{\ww}^\pmS )_{\pmS\geq 0}$ is bounded in $L^\infty L^2_{\div} (G)^3$ and $L^q \Vq ( G )^3 \cap \Xq ( G )  $, which is compactly embedded in $ L^{2 \sigma_0} (G)$ for some $2\sigma_0 < 5q/3$. 
From \eqref{eq:local-pressure:2}-\eqref{eq:local-pressure:2bis}, the same properties hold for $\nabla p_h^{\pmS}$ in any $V \Ksub G$. 
Thus, we can  apply the Aubin-Lions compactness Lemma, and we have for some $\sigma_0 < 5q/6$ a not relabeled subsequence such that:
\begin{align}
\label{eq:estim:blood:wn:2}
\widetilde{\ww}^{\pmS} 
+ \nabla p_h^{\pmS}
&\xrightarrow[\quad]{} \widetilde{\ww} +  \nabla p_h
\quad
\text{ strongly in } L^{ 2\sigma_0 } (  J \times V  )^3 , 
\end{align}
noticing that the weak and the pointwise limits being the same in $L^r$ for $1<r<\infty$. 
Now, from the strong convergence of $\vel^\pmS$
given by Definition~\ref{def:cvg:dom-veloc},  
we can write, as it is done in \cite{bucur_feireisl_necasova_wolf2008asymp}, for $\bd{\psi} \in C^{\infty}_{0,\div_x} ( J \times G )$: 
\begin{align*}
\int_J \int_G \mathbf{H}_{\vel} : \nabla \bd{\psi} \dm x \dm t 
&= \lim_{k \rightarrow \infty}  
\int_J \int_G ( \widetilde{\ww}^\pmS + \vel^\pmS ) \otimes ( \widetilde{\ww}^\pmS  + \vel^\pmS ) : \nabla \bd{\psi} \dm x \dm t \\
& = \lim_{k \rightarrow \infty}  
\int_J \int_G ( \widetilde{\ww}^\pmS + \vel^\pmS + \nabla p_h^{\pmS} ) \otimes ( \widetilde{\ww}^\pmS  + \vel^\pmS ) : \nabla \bd{\psi} \dm x \dm t \\
& \ \  - \lim_{k \rightarrow \infty}  
\int_J \int_G \nabla p_h^{\pmS}\otimes ( \widetilde{\ww}^\pmS  + \vel^\pmS  + \nabla p_h^{\pmS} ) : \nabla \bd{\psi} \dm x \dm t \\
& \ \ + \lim_{k \rightarrow \infty}  
\int_J \int_G \nabla p_h^{\pmS} \otimes \nabla p_h^{\pmS} : \nabla \bd{\psi} \dm x \dm t . 
\end{align*} 
The last term cancels because $p_h^{\pmS}$ is harmonic and $\bd{\psi}$ is solenoidal, while the other terms converge to the expected limit, leading to \eqref{eq:local-pressure:3}.  
\end{proof}

Now we can move on to identifying the convective term in \eqref{eq:wn:momentum:lim:1}. 
Let $\bd{\eta} \in \espX_p$ such that $\supp \bd{\eta} \subset \overline{ Q } \setminus \Sigma$. 
Because the support is compactly included in $\overline{ Q } \setminus \Sigma$, there exits times $0 = t_0 < t_1 < \cdots < t_{2N - 1} < t_{2N} = T$ for some $N \in \NN$, and open sets $G_i \Ksub K_i \Ksub \had$ for $i=0 , \cdots , 2N-2$, such that: 
\begin{equation*}
\supp \bd{\eta} 
\subset \bigcup_{i=0}^{2N-2} J_i \times {G_i} 
\subset
 \bigcup_{i=0}^{2N-2} J_i \times \overline{K_i} 
\subset Q ,
\quad 
\text{where }
J_i := (t_i , t_{i+2}).
\end{equation*}
We also recall that from Lemma \ref{thm:lem:compact-in-Q}, there is a $\pmS_0 \in \NN$ such that for all $\pmS \geq \pmS_0$:
\begin{equation*}
 \bigcup_{i=0}^{2N-2} J_i \times \overline{K_i} 
\subset Q^{\pmS} .
\end{equation*}
For all $i = 0 , \cdots , 2N -2$, we have that for all $\pmS \geq \pmS_0$, $\widetilde{\ww}^\pmS$ satisfies the assumptions of Proposition~\ref{thm:prop:local-pressure} written for $J_i$, $G_i$, and $K_i$, with $\FF^{\pmS}$ given by \eqref{eq:wn:momentum}. 
Then we can deduce from Proposition~\ref{thm:prop:identif:convect} that we have \eqref{eq:local-pressure:3} for $J_i \times G_i$.  
For retrieving the initial data in \eqref{eq:wn:momentum:lim:1}, we need to be more precise.   Let us start to define the interval: 
\begin{equation*}
J_{-1} := (- t_{1} , t_{1}).
\end{equation*} 
We show the following.
\begin{lemma}
\label{thm:lem:local-pressure}
It holds $\forall \bd{\psi} \in C^{\infty}_{0,\div_x} ( [0 ,t_1 )  )  \times G_0 )$: 
\begin{equation}
\label{eq:local-pressure:4}
\int_0^{t_1} \int_{G_0} \mathbf{H}_{\vel} : \nabla \bd{\psi} \dm x \dm t 
= 
\int_0^{t_1} \int_{G_0} ( \widetilde{\ww} + \vel ) \otimes ( \widetilde{\ww}   + \vel  )  : \nabla \bd{\psi} \dm x \dm t . 
\end{equation}
\end{lemma}

\begin{proof}
For any function $f \in L^1 ( (0,t_1) \times \had ) $, we define the associated symmetric and skew symmetric time reflected functions $f_{\R}$ and $f_{-\R}$ in $ \in L^1 ( (-t_1,t_1) \times \had )$, given respectively by: 
\begin{equation*}
f_{\R} ( t , x  ) := 
\left\lbrace
\begin{aligned}
& f ( t , x  ) && \text{ if } t \geq 0 , \\
& f ( -t , x  ) && \text{ if } t < 0 , \\
\end{aligned}
\right.
\quad \text{ and } \quad
f_{-\R} ( t , x  ) := 
\left\lbrace
\begin{aligned}
& f ( t , x  ) && \text{ if } t \geq 0 , \\
& -f ( -t , x  ) && \text{ if } t < 0 . \\
\end{aligned}
\right.
\end{equation*}
Considering the time reflected velocity fields $\widetilde{\ww}^\pmS_{\R}$, one can easily check that: 
$$
\widetilde{\ww}^\pmS_{\R} \in L^\infty ( J_{-1} , L^2_{\div} (K_0 )) \cap L^q ( J_{-1} , \Xq ( K_0 ) )
,
$$
and $\widetilde{\ww}^\pmS_{\R} \rightharpoonup \widetilde{\ww}_{\R}$ weakly in  
$L^{ 5q / 3 } (  J_{-1} \times K_0  )^{3 \times 3}$ 
from \eqref{eq:weaklim:tild:w:1},  so that 
$[ ( \widetilde{\ww}^\pmS + \vel^\pmS ) \otimes ( \widetilde{\ww}^\pmS  + \vel^\pmS )]_{\R} 
\rightharpoonup 
\mathbf{H}_{\vel, \R}$ weakly in  
$L^{ 5q / 6 } (  J_{-1}  \times K_0 )^{3 \times 3}$ 
from \eqref{eq:weaklim:tild:w:2}.  
We also check that:
\begin{equation*}
\int_{J_{-1} \times K_0} \widetilde{\ww}^\pmS_{\R}   \cdot \partial_t \bar{\bd{\psi}}  \dm x \dm t 
= \int_{J_{-1}} \langle \FF^{\pmS}_{-\R} (t) ,  \bar{\bd{\psi}} (t  ) \rangle_{W^{1,p}_0(K_0)} \dm t ,
\quad
\forall \bar{\bd{\psi}} \in C^{\infty}_{0,\div_x} ( J_{-1} \times K_0 ) ,
\end{equation*}
the latter expression being obtained from \eqref{eq:wn:momentum}, from the definition of the time reflections, and by change of variable in time.   
Here we also have  that $\FF^{\pmS}_{-\R} \in (L^p ( J_{-1} , W_0^{1,p} (K_0)^{3})'$. 
Thus, assumptions of Propositions~\ref{thm:prop:local-pressure} and \ref{thm:prop:identif:convect} are satisfied, so that we can apply the same procedure following the proof of Propositions~\ref{thm:prop:identif:convect} to obtain:
\begin{equation*}
\int_{J_{-1}} \int_{G_0} \mathbf{H}_{\vel, \R} : \nabla \bar{\bd{\psi}} \dm x \dm t 
= 
\int_J \int_G [( \widetilde{\ww} + \vel ) \otimes ( \widetilde{\ww}   + \vel  ) ]_{\R} : \nabla \bar{\bd{\psi}}  \dm x \dm t , 
\quad 
\forall \bar{\bd{\psi}} \in C^{\infty}_{0,\div_x} ( J_{-1} \times G_0 ). 
\end{equation*}
Finally, for $\bd{\psi} \in C^{\infty}_{0,\div_x} ( [0 ,t_1 )  )  \times G )$, we can consider approximations of $\bd{\psi}_{\R}$ in  $C^{\infty}_{0,\div_x} ( J_{-1} \times G_0 )$, so that passing to the limit in the previous expression, we obtain \eqref{eq:local-pressure:4} by symmetry of the integrands. 
\end{proof}

\medskip

We have now all the ingredients to identify the convective limit. 
Let $( \chi_i (t) )_{-1 \leq i \leq 2N-2}$ be a partition of unity of $[0,T]$ subordinated to the family $( J_i  )_{-1 \leq i \leq 2N-2}$, 
and let us redefine the time intervals $J_i' := J_i$ for $i = 0 , \cdots , 2N - 2$, and $J'_{-1} := [0,t_1)$. 
Then, recalling that $\supp \bd{\eta} 
\subset \bigcup_{i=0}^{2N-2} J_i \times {G_i} $, we have:
\begin{align*}
\int_0^T \langle \mathbf{H}_{\vel}  , \bd{\eta} \rangle_{W^{1,p}( \Omega_t )}  
&:=  \int_0^T \int_{\Omega_t }  \mathbf{H}_{\vel}  : \nabla \bd{\eta} \dm x \dm t \\
&=  \int_0^T \int_{\Omega_t } \Big( \sum_{i=-1}^{2N-2} \chi_i \Big) \mathbf{H}_{\vel}  : \nabla \bd{\eta} \dm x \dm t \\
&= \sum_{i=-1}^{2N-2} \int_{J_i'} \int_{G_i}  \mathbf{H}_{\vel}  : \nabla ( \chi_i \bd{\eta} ) \dm x \dm t ,
\end{align*}
where we have set $G_{-1} := G_0$. 
Noticing that for $i=0, \cdots , 2N-2$, we have that $\chi_i \bd{\eta}$ are  suitable test functions for applying \eqref{eq:local-pressure:3} in $J_i' \times G_i$, and using Lemma~\ref{thm:lem:local-pressure} for $\chi_{-1} \bd{\eta} $ in $J_{-1}' \times G_{-1}$, we can conclude that:
\begin{align}
\int_0^T \langle \mathbf{H}_{\vel}  , \bd{\eta} \rangle_{W^{1,p}( \Omega_t )}  
&= \sum_{i=-1}^{2N-2} \int_{J_i'} \int_{G_i}  ( \widetilde{\ww} + \vel ) \otimes ( \widetilde{\ww}   + \vel  )  : \nabla ( \chi_i \bd{\eta} ) \dm x \dm t  \notag \\
& = \int_0^T \int_{\Omega_t } ( \widetilde{\ww} + \vel ) \otimes ( \widetilde{\ww}   + \vel  ) : \nabla \bd{\eta} \dm x \dm t . 
\label{eq:local-pressure:5}
\end{align}
For any test function as above, this time cutting and local pressure representation can be applied. 
Thus we have  finally shown that for $\bd{\eta} \in \espX_p$ such that $\supp \bd{\eta} \subset \overline{ Q } \setminus \Sigma$, it holds:
\begin{multline}
\label{eq:wn:momentum:lim:2}
- \int_0^T ( \widetilde{\ww}  , \partial_t \bd{\eta} )_{L^2 ( \Omega_t )} 
+ \int_0^T \langle \bd{\chi}_{\vel} , \bd{\eta} \rangle_{W^{1,p}( \Omega_t )}
+ \int_0^T \langle \Bg_{\vel} (\widetilde{\ww})  , \bd{\eta} \rangle_{W^{1,p}( \Omega_t )}  \\
=
\int_0^T \langle \bd{f} - \partial_t \vel , \bd{\eta} \rangle_{W^{1,p}( \Omega_t )}
+ ( \widetilde{\ww}_{0} , \bd{\eta}(0) )_{L^2(\Omega)}  .
\end{multline}

\subsubsection{Identification of the stress}
\label{sec:identif-stress}

We can apply the exact same procedure as in \cite[Sec.~5.2]{nagele_ruzicka2018}. 
Let us give some details. 
We start rewriting the equation in term of $\extu = \widetilde{\ww} + \vel$. 
From \eqref{eq:wn:momentum:lim:2}, if we denote by $\bd{\chi}$ the weak limit of $\stress ( \Gs \extun )$ in $L^{ q' }  (  \Qhad  )^{3 \times 3} $, 
then we have for all $\bd{\eta} \in \espX_p$ such that $\supp \bd{\eta} \subset \overline{ Q } \setminus \Sigma$:
\begin{multline}
\label{eq:wn:momentum:lim:3}
- \int_0^T ( \uu  , \partial_t \bd{\eta} )_{ L^2 (\Omega_t) } 
+ \int_0^T \langle \bd{\chi}  , \bd{\eta} \rangle_{W^{1,p}(\Omega_t)}
+ \int_0^T ( \Bg  ( \uu )  , \bd{\eta} \rangle_{W^{1,p}(\Omega_t)}  \\
=
\int_0^T ( \bd{f} , \bd{\eta} \rangle_{W^{1,p}(\Omega_t)}
+ ( \uu_{0} , \bd{\eta}(0) )_{ L^2 (\Omega)}  .
\end{multline}
Now we define $\hh_\pmS := \uun - \uu$, 
$\bd{H}^1_\pmS := \bd{\chi} - \stress ( \Gs \uun )$, 
$\bd{H}^2_\pmS := \uu \otimes \uu - \uun \otimes \uun$, and
$\bd{H}_\pmS := \bd{H}^1_\pmS + \bd{H}^2_\pmS$. 
From \eqref{eq:wn:momentum:lim:3} and \eqref{pbm:blood:weak:moving-dom:1}, we have for any $\bd{\eta} \in C^\infty_{0, \div} (Q)^3$, and for all $\pmS$ large enough:
\begin{equation*}
- \int_Q  \hh_\pmS \cdot \partial_t \bd{\eta} =  \int_Q  \bd{H}_\pmS : \gd \bd{\eta} .
\end{equation*}
Let $J \subset (0,T)$ be an interval and $B \subset \RR^{\dn}$ be a ball such that $Q_0 := J \times B \Ksub Q$. 
In \cite[Thm. 63, Cor. 64]{nagele_ruzicka2018}, it is shown that if: 
\begin{align*}
\hh_\pmS & \xrightharpoonup[\quad]{}  
\bd{0}
&& \hspace{-1ex} \text{ weakly in } L^{ q } ( J , \Xq (  B  )^3 ) 
,  
&&& \bd{H}^1_\pmS &\xrightharpoonup[\quad]{} \bd{0} 
&& \hspace{-1ex} \text{ weakly in } L^{ q' }  (  Q_0  )^{3 \times 3} , \\
\hh_\pmS & \xrightarrow[\quad]{}  
\bd{0}
&& \hspace{-1ex} \text{ strongly in } L^{ 2 \sigma_0 } ( Q_0 ) ,  
&&& \bd{H}^2_\pmS &\xrightarrow[\quad]{} \bd{0}
&& \hspace{-1ex} \text{ strongly in } L^{  \sigma_0 }  (  Q_0 )^{3 \times 3} , 
\end{align*}
then it holds: 
$
\lim_{n \rightarrow \infty}  \Gs ( \uun ) = \Gs ( \uu  )
$ 
almost everywhere in $\frac{1}{8}Q_0$, so that finally: 
\begin{equation}
\label{eq:Dun:lim.a.e.Q}
\lim_{n \rightarrow \infty}  \Gs ( \uun ) = \Gs ( \uu  )
\quad
\text{ almost everywhere in $Q$. }
\end{equation}
From what has been done, $\hh_\pmS$, $\bd{H}^1_\pmS$, and $\bd{H}^2_\pmS$ satisfy these assumptions, and the convergence \eqref{eq:Dun:lim.a.e.Q} is then obtained. 
The almost everywhere convergence also holds for $\stress ( \Gs \extun )$. 
Thus, the weak and the pointwise limits being the same in $L^r$ for $1<r<\infty$,  
one gets that $\stress ( \Gs \extu ) =  \bd{\chi}$ a.e. in $Q$.

\subsubsection{Density of compactly supported test functions}

Let us define the space    
$ 
\label{eq:def:dens:Smoothtime:Kspace}
\X ( Q ) := 
\lbrace 
\bd{\varphi} \in  C_0^\infty ( [0 , T) , C_{0,\div}^\infty ( D )^3 )  
 \mid
\supp ( \bd{\varphi} ) \Ksub Q
\rbrace  
$.   
We have the following density result (see Appendix \ref{sec:density} for a proof).

\begin{proposition}
\label{thm:density:H1L2capLqVq} 
Let $2 \leq p < \infty$, and $\bd{\eta} \in \espX_p$. 
Then there exists a sequence $(\bd{\eta}_\pmA)_{\pmA \geq 0} \subset \X ( Q )$, such that: 
\begin{align*}
\bd{\eta}_\pmA
\underset{\pmA\rightarrow \infty}{\rightarrow}
\bd{\eta} \text{ strongly in } L^p ( (0,T) ,  \Vp ( \Omega_t )^3 ) , \quad
\partial_t  
\bd{\eta}_\pmA
\underset{\pmA\rightarrow \infty}{\rightharpoonup}
\partial_t \bd{\eta} \text{ weakly in } L^2 ( (0,T) , L^2 ( \Omega_t )^3 )  .
\end{align*}
\end{proposition}

In view of the above computations, we have obtained in particular that:
\begin{multline*} 
\label{eq:wn:momentum:lim:4}
- \int_0^T ( \uu  , \partial_t \bd{\eta} )_{ L^2 (\Omega_t) } 
+ \int_0^T \langle \stress ( \Gs \uu )  , \bd{\eta} \rangle_{W^{1,p}(\Omega_t)}
+ \int_0^T ( \Bg  ( \uu )  , \bd{\eta} \rangle_{W^{1,p}(\Omega_t)}  \\
=
\int_0^T ( \bd{f} , \bd{\eta} \rangle_{W^{1,p}(\Omega_t)}
+ ( \uu_{0} , \bd{\eta}(0) )_{ L^2 (\Omega)} , 
\quad 
\forall  \bd{\eta}   \in   \X ( Q ) .
\end{multline*}
Finally, from Proposition \ref{thm:density:H1L2capLqVq}, 
we can conclude that $\uu$ is a weak solution to Problem~\eqref{pbm:blood_flow:hetero-Dirichlet:moving-dom}-\eqref{eq:def:stress:2} written on $Q$. Then Theorem \ref{thm:shapo-cont:blood} is proved.

%#####################################################################
%#####################################################################
\section{Existence of a minimal shape}
\label{sec:exist-min}

We apply the direct method of calculus of variations.

\begin{proof}
In view of assumption \ref{item:assump:j:1}, we have that $\SHPfun ( \Omega , \vel  )$ is well-defined and non-negative for all $( \Omega , \vel  ) \in \Opn \times \V$. 
Because $\Opn \times \V$ is not empty, we thus have that 
$
\SOinf := \inf_{( \Omega , \vel  ) \in \Opn \times \V} \lbrace \SHPfun ( \Omega , \vel  ) \rbrace 
$   
exists.  
Let $(( \Omega^\pmS , \vel^\pmS ))_{\pmS\geq 0} \subset \Opn \times \V$ be a minimizing sequence  for $\SHPfun$.  
From Proposition \ref{thm:compactness:geom}, we can consider a not relabeled subsequence $(( \Omega^\pmS , \vel^\pmS ))_{\pmS\geq 0}$ converging to some $(\Omega^* , \vel^*) \in \Opn \times \V$ in the sense of Definition \ref{def:cvg:dom-veloc}. 
Let $\pmE \geq 1$.
From Theorem~\ref{thm:shapo-cont:blood}, and from the definition \eqref{eq:def:SHPfun} of $\SHPfun$, we have that for all $\pmS \in \NN$, there exists $\uun_\pmE \in \nP ( \Omega^\pmS , \vel^\pmS )$, such that: 
\begin{equation}
\label{eq:pf:existenceSO:1}
\SHPfun ( \Omega^\pmS , \vel^\pmS )
\leq 
\shpFUN ( \Omega^\pmS , \vel^\pmS ; \uun_\pmE  ) 
<  
\SHPfun ( \Omega^\pmS , \vel^\pmS ) + \frac{1}{\pmE} .
\end{equation}
Still from Theorem~\ref{thm:shapo-cont:blood}, let $( \uun_\pmE )_{\pmS\geq 0}$  be a not relabeled subsequence converging to some $\uu_\pmE  \in \nP (\Omega^* , \vel^*)$. 
Letting $\pmS$ go to $\infty$ in \eqref{eq:pf:existenceSO:1}, recalling that $( \Omega^\pmS , \vel^\pmS )$ is a minimizing sequence, and taking into account the lower semicontinuity property \ref{item:assump:j:2} of $\shpFUN$, we can write:
\begin{equation*}
\SOinf
 \leq
\SHPfun ( \Omega^* , \vel^* )
\leq
\shpFUN ( \Omega^* , \vel^* ; \uu_\pmE  ) 
 \leq 
 \liminf_{\pmS\rightarrow\infty}
\shpFUN ( \Omega^\pmS , \vel^\pmS ; \uun_\pmE  )  
\leq 
\SOinf  + \frac{1}{\pmE} .
\end{equation*} 
Letting $\pmE \rightarrow \infty$, we obtain  
$\SHPfun ( \Omega^* , \vel^* ) = \SOinf$, 
so that $( \Omega^* , \vel^* ) $ is a solution to Problem \eqref{eq:pbm:minimization:abstract}.
\end{proof}

%#####################################################################
%#####################################################################
\section{Conclusion}
\label{sec:ccl}

We have shown the existence of minimal shapes for blood flows in  
moving domains, 
by proving the shape continuity of the velocity solutions.  
Blood is described by generalized Navier-Stokes equations (power law $q > 6/5$ modeling shear-thinning fluids).
This model allows for considering  prescribed movements of the solid domain surrounding and inside the blood (ex: blood pump modeling), 
and to tackle the hemolysis index minimization problem (see Example~\ref{thm:ex:hemolysis}).    

 The blood velocity is here assumed to fit the one of the domain on its boundary.  
 To go further, one could consider more general BCs  
(ex: prescribed velocity on the inlet; Navier type BCs,  
see \cite{bsaies_dziri2020sh-sens_navier-bc, boukrouche_et-al2022non-newt_friction}); 
artificial BCs  
on the outlet).
It could also be interesting to take into account the elasticity of blood vessels,   carrying out the same kind of analysis based on the work of \cite{lengeler2014gnl-newt_koiter}.

Although our analysis has led to an existence result, the non-uniqueness of the solutions prevents  from studying the shape differentiability of the problem. 
This is an impediment to any numerical study based on shape gradient method. 
In order to circumvent this lack of uniqueness, to compute shape derivatives, and to tackle numerically the blood damage minimization, one could work with strong solutions (unique locally in time), or with 2D Newtonian fluids. 
Otherwise, regardless of the model chosen, the analysis of the convergence of a numerical scheme would then be a fundamental issue.

Finally, concerning the hemolysis index in the framework of numerical computation, a more practical way of modeling would be to consider transport equation (see, e.g. \cite{nam2011stlsfem, yu2017review}). 
Thus, as a first step, it would be interesting to consider the same shape continuity result for transport equations, 
and then to study the shape sensitivity of transport equations.

\appendix

%*********************************************************************
%*********************************************************************
\section{Density results}
\label{sec:density}

\subsection{Preliminary results}

Let $\Omega \subset \RR^{\dn}$ be a bounded open set, 
having the $c_{\Opn}$-cone property.  
We recall that $\Omega$ has then a Lipschitz boundary, and its Lipschitz character $\Lip (\Omega)$ depends only on $c_{\Opn}$ (see \cite[Rem. 2.4.8]{henrot_pierre2018shapo_geom}). 
Let $f \in C^\infty_0 ( \Omega  )$, such that  
$ 
\label{eq:div-pbm:compatib-cond}
\int_{\Omega } f (x) \dm x = 0
$.    
There exists a linear map $\bogov$, called \textit{Bogovski\u{i}'s operator}, acting on such functions $f$ and taking its values in $C^\infty_0 ( \Omega  )$ such that:
\begin{equation*}
\label{eq:div-pbm}
 \div \bogov [f] = f \text{ on } \Omega .
\end{equation*}
Let us compile some further needed properties of the Bogovski\u{i}'s operator that can be found in \cite[Sec.  10.5]{feireisl_novotny2017singular_book}. 
For recent progress on the topic in moving domains, see \cite{saari2023construction}.  
\begin{proposition}
\label{thm:bogov}
Let $u \in L^p(\Omega)$ such that $\int_{\Omega} u (x) \dm x = 0$, 
$\uu \in W_0^{1,p} (\Omega)^3$, 
and $v \in W^{1,p} ( (0,T) , L^p (\Omega ) )$ such that $\int_{\Omega} v (t,\cdot) \dm x = 0$ for a.e. $t \in (0,T)$, 
for $1 < p < \infty$. 
Then: 
\begin{align}
\label{eq:estim:bogov:div-1}
\lV \bogov [ u ] \rV_{W^{1,p}} 
\leq c_{\bogov} (\Omega)
\lV u  \rV_{L^p}  ,  \qquad 
\lV \bogov [ \div \uu ] \rV_{L^p} 
\leq c_{\bogov} (\Omega)
\lV \uu  \rV_{L^p}  ,
\end{align}   
and for a.e.$(t,x) \in (0,T) \times \Omega$:
\begin{equation}
\label{eq:bogov:commut:dt}
\partial_t \bogov [ v ] (t,x) 
= \bogov [ \partial_t v ] (t,x) .
\end{equation}
\end{proposition}

Let $1< p < \infty$, 
and $c_{\bogov} (\Omega)$ be the so-called \emph{Bogovski\u{i}'s constant} that appears in Proposition~\ref{thm:bogov}. 
Because $\partial \Omega$ is Lipschitz, there exists a so-called \emph{Korn's constant} denoted by $c_{ \korn , p } ( \Omega )$ such that  
for any vector field $\vv \in H^{1} (\Omega)^3$:
\begin{equation*}
\lV \gd \vv \rV_{L^{p} (\Omega)}
\leq c_{ \korn , p } ( \Omega  )
\lV \Gs \vv \rV_{L^{p} (\Omega)}.
\end{equation*}
It is known that $c_{\bogov} (\Omega)$ depends only on $p$, the dimension, and $\Lip (\Omega)$, and from there, it is shown in \cite{bucur_feireisl_necasova2008influ} that $c_{ \korn , p } ( \Omega  )$ also depends only on $p$, dimension, and $\Lip (\Omega)$. 
Thus, by definition \eqref{eq:def:class:O}, for $1< p < \infty$, there exist two constants $c_{\bogov} > 0$ and $c_{ \korn } > 0$ such that for any $\Omega \in \Opn$: 
\begin{equation}
c_{\bogov} (\Omega) \leq c_{\bogov} ,
\quad \text{ and } \quad
c_{ \korn , p } ( \Omega  ) \leq c_{ \korn } .
\label{eq:unif:Bogv-and-Korn}
\end{equation}
Furthermore, 
let $\vel \in \V$ and  
$\bd{\varphi}$ be given by \eqref{eq:def:transfo-assoc-velocity}. 
For all $t\in [0,T]$, we have  
$
\Lip (\bd{\varphi}(t,\cdot) , \had) \leq \exp ( c_{\V} T )
$, where: 
\begin{equation*}
\Lip (\bd{\varphi}(t,\cdot) , \had) := \sup 
\left\lbrace
\frac{\lv \bd{\varphi}(t, x ) - \bd{\varphi}(t, y ) \rv}{ \lv x - y \rv } \;\Big\vert\; x,y \in \had, \ x\neq y 
\right\rbrace .
\end{equation*}
We also have that $\bd{\varphi}(t,\cdot)$ is invertible, and its inverse $\bd{\varphi}^{-1}(t,\cdot) $ satisfies  
$
\Lip (\bd{\varphi}^{-1}(t,\cdot) , \had) \leq \exp ( c_{\V} T )
$. 
Thus, from \cite[Thm. 4.1]{hofmann_mitrea_taylor2007geom-dom}, we have that for all $t\in [0,T]$,  
$\Lip (\Omega_t)$, is controlled in terms of $\Lip (\bd{\varphi}(t,\cdot) , \had)$, 
$\Lip (\bd{\varphi}^{-1}(t,\cdot) , \had)$, and
and $\Lip (\Omega)$.  
Then, there exists also two constants still denoted by $c_{\bogov} > 0$ and $c_{ \korn } > 0$ depending on $c_{\Opn}$ and $c_{\V}$ such that for any $(\Omega , \vel )   \subset  \Opn \times \V$, and for all $t\in [0,T]$: 
\begin{equation}
c_{\bogov} (\Omega_t) \leq c_{\bogov} ,
\quad \text{ and } \quad
c_{ \korn , p } ( \Omega_t  ) \leq c_{ \korn } .
\label{eq:unif:Bogv-and-Korn_t}
\end{equation}

\smallskip

We will need the following Lemma, 
which summarizes some of the results that can be found in \cite{nagele2017} (see Lemma 3.7, equation (3.24), and Proposition 3.48).

\begin{lemma}
\label{thm:lem:pull-back:piola}
Let $\bd{\varphi}$ defined by \eqref{eq:def:transfo-assoc-velocity} be a smooth diffeomorphism from $\Omega$ to $\Omega_t$. 
We define  
the \emph{Piola transform} $\mathcal{P}_{\bd{\varphi}}$ as follows. For any $\uu \in L^1 ( Q )^3$, and for a.e. $(t,x) \in [0,T]\times \Omega$:
\begin{align*}
(\mathcal{P}_{\bd{\varphi}} \uu )(t,x)  := \det ( \nabla \bd{\varphi} ( t , x ) ) (\nabla \bd{\varphi} ( t , x ))^{-1}  \uu ( t , \varphi ( t , x ) ).
\end{align*}
Defined this way, 
the Piola transform extends to an  isomorphism: 
\begin{equation} 
\mathcal{P}_{\bd{\varphi}} : L^p ( (0,T), B(t)  ) \longrightarrow L^p ( (0,T), B(0)  ),
\end{equation} 
where $B(t)$ denotes either $L^p (\Omega_t)^3$ 
or $W^{1,p}_{0,\div}  (\Omega_t)^3$, recalling that $\Omega_0 = \Omega$, 
and where the inverse is nothing else than $\mathcal{P}_{\bd{\varphi^{-1}}}$. 
Furthermore, let $\bd{\eta}\in H^1 ( (0,T) , L^2 ( \Omega_t )^3 ) \cap L^p ( (0,T) , \Vp ( \Omega_t )^3  )$ for some $p \geq 2$.  
Then we have:
\begin{equation*}
 \mathcal{P}_{\bd{\varphi}} \bd{\eta} \in H^1 ( (0,T) , L^2 ( \Omega )^3 ) \cap L^p ( (0,T) , \Vp ( \Omega )^3 ),
\end{equation*}
and there exists a constant $C>0$ such that:
\begin{equation*}
\lV \partial_t ( \mathcal{P}_{\bd{\varphi}} \bd{\eta} ) \rV_{L^2 L^2( \Omega )}
\leq C
\left( 
  \lV \bd{\eta}  \rV_{H^1 L^2( \Omega_t )}
+ \lV  \gd_x \bd{\eta} \rV_{L^2 L^2( \Omega_t )}
\right) .
\end{equation*}
\end{lemma}

From the above results, 
we can prove the density result.

\subsection{Proof of Proposition \ref{thm:density:H1L2capLqVq}}

Let $\bd{\eta} \in \espX_p$.

\noindent$\bullet$ \emph{Construction of smooth approximations.}
We first consider its Piola transform $ \mathcal{P}_{\bd{\varphi}} \bd{\eta}$ which, in view of definition \eqref{eq:def:space:XqD} and Lemma~\ref{thm:lem:pull-back:piola}, belongs to 
$ 
H^1 ( (0,T) , L^2 (\Omega)^3 ) \cap L^p ( (0,T) , W^{1,p}_0 ( \Omega )^3 )
$  
and is such that 
$  
\mathcal{P}_{\bd{\varphi}} \bd{\eta} ( T, \cdot ) = \bd{0}
$.  
Afterwards, we extend $\mathcal{P}_{\bd{\varphi}} \bd{\eta} $ in space to $ L^p ( (0,T) , \Vp (\had)^3 )$ by $\bd{0}$, then we extend it in time to $L^p ( (0,\infty) , \Vp (\had)^3 )$ by $\bd{0}$, and then we extend it in time to $L^p ( \RR , \Vp (\had)^3 )$ by reflection at $t=0$, denoting by $\overline{\mathcal{P}_{\bd{\varphi}} \bd{\eta} }$ the finally extended field. 
We have: 
\begin{equation}
\label{eq:densite:proof:-1}
\overline{\mathcal{P}_{\bd{\varphi}} \bd{\eta} } \in
H^1 ( \RR , L^2 (\had)^3 ) \cap L^p ( \RR , W^{1,p}_0 ( \had )^3 ) , 
\end{equation}
with $\supp (\overline{\mathcal{P}_{\bd{\varphi}} \bd{\eta} })  \subset [-T,T]\times \overline{\Omega}$.
Now we set:
\begin{equation}
\label{eq:densite:proof:3}
\widetilde{\bd{\eta}}_\pmA :=  \omega_{\pmA} \conv ( \xi_{\pmA} \overline{\mathcal{P}_{\bd{\varphi}} \bd{\eta} } ) ,
\end{equation}
where for $\pmA$ large enough, $\xi_{\pmA} $ is a smooth cut-off function in $C^\infty ( \RR \times \RR^{ \dn } )$ such that:

\smallskip

\begin{enumerate}[label=(\roman*)]
\item
\label{item:proof:dens:1}
 $\supp ( \xi_{\pmA} ) \subset (- T + b / \pmA  , T - b / \pmA  ) \times \Omega_{(b/\pmA)}$, 
\item 
\label{item:proof:dens:2}
$\xi_{\pmA} \equiv 1$ on $(- T + a / \pmA  , T - a / \pmA  ) \times \Omega_{(a/\pmA)}$, 
\item 
\label{item:proof:dens:3}
and $\rV \partial_t \xi_{\pmA} \rV_{L^\infty} + \rV \gd_x \xi_{\pmA} \rV_{L^\infty}  \leq C \pmA$, 
\end{enumerate}
\smallskip

\noindent 
for some $1 < b < a$ and $C>0$,  
and where $\omega_{\pmA} \in  C^{\infty}_0 ( \RR \times \RR^{\dn}) $ is a mollifier such that $\supp ( \omega_{\pmA} ) \subset B_{\RR\times\RR^{\dn}}( 0 , c/\pmA)$ for some $0< c < b-1$. 
We recall that the definition of $\Omega_{(a/\pmA)}$ and $\Omega_{(b/\pmA)}$ is given in \eqref{eq:def:int-dom}. 
Consequently, 
$\widetilde{\bd{\eta}}_\pmA \in C^{\infty}_0 ( (- T + 1 / \pmA  , T - 1 / \pmA  ) \times \Omega_{(1/\pmA)} )$, so that obviously 
$\div \widetilde{\bd{\eta}}_\pmA \in C^{\infty}_0 ( (- T + 1 / \pmA  , T - 1 / \pmA  ) \times \Omega_{(1/\pmA)} )$ and for all $t \in (- T + 1 / \pmA  , T - 1 / \pmA  ) $, we have 
$\int_{\Omega_{(1/\pmA)}} \div \widetilde{\bd{\eta}}_\pmA ( t, x ) \dm x = 0$. 
Thus we define:
\begin{equation*}
\overline{\bd{\eta}}_\pmA := \widetilde{\bd{\eta}}_\pmA - \bogov [ \div \widetilde{\bd{\eta}}_\pmA ] .
\end{equation*}
Finally, we consider the reverse Piola transform of $\overline{\bd{\eta}}_\pmA$ restricted to $[0,T]\times \Omega$, setting:
\begin{equation*}
\bd{\eta}_\pmA := \mathcal{P}_{\bd{\varphi^{-1}}} \overline{\bd{\eta}}_\pmA 
\in 
C^{\infty}_0 \Big( 
\bigcup_{ t \in [0 , T - 1 / \pmA  ) } \lbrace t \rbrace  \times \Omega_{t, (c_{\V} /\pmA)} 
\Big) ,
\end{equation*}
so that  
$ 
 \div \bd{\eta}_\pmA = 0
$.

\noindent $\bullet$ \emph{Convergence of smooth approximations.}
We now show that:
\begin{equation}
\label{eq:densite:proof:1}
\lV \bd{\eta}_\pmA - \bd{\eta} \rV_{H^1_{0,T} L^2( \Omega_t )} \xrightharpoonup[\pmA\rightarrow \infty]{} 0
\quad \text{ and } \quad 
\lV \bd{\eta}_\pmA - \bd{\eta} \rV_{L^p_{0,T} \Vp ( \Omega_t )} \xrightarrow[\pmA\rightarrow \infty]{} 0 .
\end{equation}
First, we notice that from Lemma \ref{thm:lem:pull-back:piola} (applied for $\mathcal{P}_{\bd{\varphi^{-1}}}$) and in view of the identity $\bd{\eta} =  \mathcal{P}_{\bd{\varphi^{-1}}} \mathcal{P}_{\bd{\varphi}} \bd{\eta}$, 
there exists a constant $C>0$ such that:
\begin{align*}
\lV \bd{\eta}_\pmA - \bd{\eta} \rV_{H^1_{0,T} L^2( \Omega_t )} 
&  \leq  C
\left( 
\lV \overline{\bd{\eta}}_\pmA  -  \mathcal{P}_{\bd{\varphi}} \bd{\eta} \rV_{H^1_{0,T} L^2( \Omega )} 
+
\lV \gd_x (\overline{\bd{\eta}}_\pmA  -  \mathcal{P}_{\bd{\varphi}} \bd{\eta} ) \rV_{L^2_{0,T} L^2 ( \Omega )} 
\right)   ,  \\
\lV \bd{\eta}_\pmA - \bd{\eta} \rV_{L^p_{0,T} \Vp( \Omega_t )} 
&  \leq  C
\lV \overline{\bd{\eta}}_\pmA  -  \mathcal{P}_{\bd{\varphi}} \bd{\eta} \rV_{L^p_{0,T} \Vp( \Omega )}  .
\end{align*}
Thus,   
proving \eqref{eq:densite:proof:1}
amounts to showing that: 
\begin{equation*}
\label{eq:densite:proof:2}
\lV \partial_t \overline{\bd{\eta}}_\pmA - \partial_t  \mathcal{P}_{\bd{\varphi}} \bd{\eta} \rV_{L^2_{0,T} L^2( \Omega )} \xrightharpoonup[\pmA\rightarrow \infty]{}  0
\quad \text{ and } \quad 
\lV \overline{\bd{\eta}}_\pmA -  \mathcal{P}_{\bd{\varphi}} \bd{\eta} \rV_{L^p_{0,T} \Vp ( \Omega )} \xrightarrow[\pmA\rightarrow \infty]{}  0 .
\end{equation*} 
We proceed in two steps. We show  \ref{item:proof:dens:4} first that $\widetilde{\bd{\eta}}_\pmA \rightarrow \mathcal{P}_{\bd{\varphi}} \bd{\eta}$, and \ref{item:proof:dens:5} second that $\bogov [ \div \widetilde{\bd{\eta}}_\pmA ] \rightarrow 0$ as $\pmA \rightarrow \infty$, in both cases for the required norms
$\lV \partial_t ( \cdot ) \rV_{L^2_{0,T} L^2( \Omega )}$ weak and
$\lV \cdot \rV_{L^p_{0,T} \Vp( \Omega )}$ strong. 
\begin{enumerate}[label=(\arabic*), leftmargin=4.5ex]
\item 
\label{item:proof:dens:4}
\textit{Convergence of the main term $\widetilde{\bd{\eta}}_\pmA$.}  
From \eqref{eq:densite:proof:3}, we  
have $\lV \widetilde{\bd{\eta}}_\pmA -  \mathcal{P}_{\bd{\varphi}} \bd{\eta} \rV_{L^p L^p(t)} \longrightarrow 0 $ as $\pmA\rightarrow \infty$. 
From \eqref{eq:densite:proof:-1} and because the supports of the involved functions are compactly supported in $[-T,T]\times \Omega$, we can interchange the time derivative and the convolution.
Thus, we estimate the norm of the weak time derivative as follows:
\begin{align}
\lV \partial_t &\widetilde{\bd{\eta}}_\pmA  -  \partial_t  \mathcal{P}_{\bd{\varphi}} \bd{\eta} \rV_{L^2_{0,T} L^2( \Omega )}  \notag \\
&\leq
\big\lV  \omega_{\pmA} \conv \partial_t ( \xi_{\pmA} \overline{\mathcal{P}_{\bd{\varphi}} \bd{\eta} } -  \overline{\mathcal{P}_{\bd{\varphi}} \bd{\eta} } )   
\big\rV_{L^2_{0,T} L^2( \Omega )}   +
\big\lV  \omega_{\pmA} \conv  \partial_t \overline{\mathcal{P}_{\bd{\varphi}} \bd{\eta} }  - \partial_t \mathcal{P}_{\bd{\varphi}} \bd{\eta}  
\big\rV_{L^2_{0,T} L^2( \Omega )}   \notag \\
&\leq
\big\lV  \partial_t ( \xi_{\pmA} \overline{\mathcal{P}_{\bd{\varphi}} \bd{\eta} } -  \overline{\mathcal{P}_{\bd{\varphi}} \bd{\eta} } )   
\big\rV_{L^2_{-T,T} L^2( \Omega )} 
 +
\big\lV  \omega_{\pmA} \conv  \partial_t \overline{\mathcal{P}_{\bd{\varphi}} \bd{\eta} }  - \partial_t \overline{\mathcal{P}_{\bd{\varphi}} \bd{\eta} }
\big\rV_{L^2_{-T,T} L^2( \Omega )}  .
\label{eq:densite:proof:4}
\end{align}
We directly notice that the second term in \eqref{eq:densite:proof:4} goes to $0$ as $\pmA \rightarrow \infty$. Regarding the first term, we have:
\begin{multline}
\label{eq:densite:proof:5}
\hspace{1cm}
\big\lV  \partial_t ( ( \xi_{\pmA} - 1 ) \overline{\mathcal{P}_{\bd{\varphi}} \bd{\eta} } )   
\big\rV_{L^2_{-T,T} L^2( \Omega )}  \\
\leq  
\big\lV ( \xi_{\pmA} - 1 ) \partial_t \overline{\mathcal{P}_{\bd{\varphi}} \bd{\eta} } 
\big\rV_{L^2_{-T,T} L^2( \Omega )}
+
\big\lV  \partial_t ( \xi_{\pmA} ) \overline{\mathcal{P}_{\bd{\varphi}} \bd{\eta} } 
\big\rV_{L^2_{-T,T} L^2( \Omega )} ,
\end{multline}
where the first term of the right hand side goes to zero as $\pmA \rightarrow \infty$. 
By definition, $ \partial_t  \xi_{\pmA} = 0 $  outside $I_b \times \Omega_{ab} \cup I_{ab} \times \Omega_{a}$, where,  using definition \eqref{eq:def:int-dom}, we set: 
\begin{align*}
I_b \times \Omega_{ab} 
&:= 
\Big( -T + \frac{b}{\pmA} , T - \frac{b}{\pmA}  \Big)
\times \big( \Omega_{(b/\pmA)} \big)^{(a/\pmA)} , \\
I_{ab} \times \Omega_{a} 
&:= 
\Big( -T + \frac{b}{\pmA} , T - \frac{b}{\pmA}  \Big) \setminus 
 \Big[ -T + \frac{a}{\pmA} , T - \frac{a}{\pmA}  \Big] 
\times \Omega_{(a/\pmA)} .
\end{align*}
By definition, for any $x \in \Omega_{ab}$ and for any $t \in I_{ab}$, we have that $\pmA < a \dist ( x , \partial \Omega )^{-1}$ and $\pmA < a \dist ( t , \lbrace -T , T \rbrace  )^{-1}$. 
Then we compute the norm of the second term of the right hand side of \eqref{eq:densite:proof:5} on these two subsets respectively.
We also note that for a.e. $t\in (-T,T)$, 
$\overline{\mathcal{P}_{\bd{\varphi}} \bd{\eta} } (t , \cdot ) \in W^{1,2}_0 (\Omega )$, 
and for a.e. $x \in \Omega$, 
$\overline{\mathcal{P}_{\bd{\varphi}} \bd{\eta} } ( \cdot , x ) \in W^{1,2}_0 ( (-T,T) )$, 
so that, taking also into account \ref{item:proof:dens:3} and Hardy's inequality for $\Omega$ (whose constant is denoted by $c_{\hardy_\Omega}$), we have:
\begin{align*}
\int_{I_b }  \int_{ \Omega_{ab} }
\left\lv
  \partial_t ( \xi_{\pmA} ) \overline{\mathcal{P}_{\bd{\varphi}} \bd{\eta} }
\right\rv^2 
\dm x \dm t 
& \leq 
(C a )^2 
\int_{I_b }   
\left\lV
\dist ( \cdot , \partial \Omega )^{-1} 
  \overline{\mathcal{P}_{\bd{\varphi}} \bd{\eta} }
\right\rV_{L^2 (\Omega_{ab})}^2  \dm t  \notag\\
& \leq 
(C a c_{\hardy_\Omega})^2 
\int_{I_b }   
\left\lV
\gd_x
\overline{\mathcal{P}_{\bd{\varphi}} \bd{\eta} }
\right\rV_{L^2 (\Omega)}^2  \dm t  .
\end{align*}
Thus, $\partial_t ( \xi_{\pmA} ) \overline{\mathcal{P}_{\bd{\varphi}} \bd{\eta} } \longrightarrow 0$  in $L^2 ( (-T,T), L^2 ( \Omega )^3)$ by absolute continuity of Lebesgue integral, 
because $\dist ( \cdot , \partial \Omega )^{-1} 
  \overline{\mathcal{P}_{\bd{\varphi}} \bd{\eta} } \in L^2 (\Omega)$ 
and $\lv \Omega_{ab} \rv \rightarrow 0$. The same arguments holds for the integration over $I_{ab} \times \Omega_{a}$, rewriting:
\begin{align*}
\int_{ \Omega_{a} }  \int_{ I_{ab} }
\left\lv
  \partial_t ( \xi_{\pmA} ) \overline{\mathcal{P}_{\bd{\varphi}} \bd{\eta} }
\right\rv^2 
\dm x \dm t 
& \leq 
(C a )^2 
\int_{ \Omega_{a} }   
\left\lV
\dist ( t , \lbrace -T , T \rbrace  )^{-1} 
\overline{\mathcal{P}_{\bd{\varphi}} \bd{\eta} }
\right\rV_{L^2 ( I_{ab} )}^2 \dm x  \notag\\
& \leq 
(C a c_{\hardy_T})^2 
\int_{ \Omega_{a} }   
\left\lV
\partial_t
  \overline{\mathcal{P}_{\bd{\varphi}} \bd{\eta} }
\right\rV_{L^2 ( (-T,T) )}^2  \dm x  ,
\end{align*}
where $c_{\hardy_T}$ is the Hardy's inequality constant associated to the interval $(-T,T)$. 
All terms going to $0$ as $\pmA \rightarrow \infty$ in \eqref{eq:densite:proof:5}, we finally have that 
$\lV \partial_t \widetilde{\bd{\eta}}_\pmA  -  \partial_t  \mathcal{P}_{\bd{\varphi}} \bd{\eta} \rV_{L^2_{0,T} L^2 ( \Omega )}  \rightarrow 0$. 
Following exactly the same explanations given in the present Item \ref{item:proof:dens:4}, we show as well that $\lV \gd_x \widetilde{\bd{\eta}}_\pmA  -  \gd_x \mathcal{P}_{\bd{\varphi}} \bd{\eta} \rV_{L^p_{0,T} L^p( \Omega )}  \rightarrow 0$. 
So far, we have shown that: 
\begin{align}
\label{eq:conv:tilde-eta-k:LpVp}
&\widetilde{\bd{\eta}}_\pmA \underset{\pmA\rightarrow \infty}{\longrightarrow} \mathcal{P}_{\bd{\varphi}} \bd{\eta} 
\quad\text{ strongly in } L^p ( (0,T) , \Vp( \Omega ) ) , \\
\label{eq:conv:tilde-eta-k:H1H}
&\widetilde{\bd{\eta}}_\pmA \underset{\pmA\rightarrow \infty}{\longrightarrow} \mathcal{P}_{\bd{\varphi}} \bd{\eta} 
\quad\text{ strongly in } H^1 ( (0,T) , L^2( \Omega ) ) .
\end{align}
\item 
\label{item:proof:dens:5}
\textit{Convergence of the correction term $\bogov [ \div \widetilde{\bd{\eta}}_\pmA ] $.}  
From Proposition \ref{thm:bogov}, we can write:
\begin{align*}
\lV 
\bogov [ \div \widetilde{\bd{\eta}}_\pmA ]
\rV_{L^p_{0,T} \Vp( \Omega )}  
&\leq  
c_{\bogov} ( \Omega )
\lV  
\div \widetilde{\bd{\eta}}_\pmA 
\rV_{L^p_{0,T} L^p(  \Omega  )}  ,
\end{align*}
and in view of \eqref{eq:conv:tilde-eta-k:LpVp}, $ \div \widetilde{\bd{\eta}}_\pmA  \rightarrow  \div \overline{ \mathcal{P}_{\bd{\varphi}} \bd{\eta} } = 0 $ in $L^p( (0,T) , L^p( \Omega ))$, so that $\bogov [ \div \widetilde{\bd{\eta}}_\pmA ] \rightarrow 0$ in $L^p( (0,T) , \Vp( \Omega )^3)$. Then from \eqref{eq:estim:bogov:div-1} and \eqref{eq:bogov:commut:dt}:
\begin{align*}
\lV \partial_t  
\bogov [ \div \widetilde{\bd{\eta}}_\pmA ]
\rV_{L^2_{0,T} L^2( \Omega )}  
&= 
\lV   
\bogov [ \div \partial_t \widetilde{\bd{\eta}}_\pmA ]
\rV_{L^2_{0,T} L^2( \Omega )}  \\
&\leq  c_{\bogov} ( \Omega )
\lV  
\partial_t  \widetilde{\bd{\eta}}_\pmA 
\rV_{L^2_{0,T} L^2( \Omega )} .
\end{align*}
Thus, in view of \eqref{eq:conv:tilde-eta-k:H1H}, $\partial_t  
\bogov [ \div \widetilde{\bd{\eta}}_\pmA ]$ is bounded in $L^2 ( (0,T) , L^2( \Omega )^3 )$. 
Let $\bd{\varphi} \in (C^{\infty}_0 ( (0,T)) \times \Omega )^3$, we have: 
\begin{equation*}
\int_{(0,T) \times \Omega } 
\partial_t  
\bogov [ \div \widetilde{\bd{\eta}}_\pmA ]
\cdot
\bd{\varphi} \dm t \dm x 
=  -
\int_{(0,T) \times \Omega } 
\bogov [ \div \widetilde{\bd{\eta}}_\pmA ]
\cdot 
\partial_t  
\bd{\varphi} \dm t \dm x  ,
\end{equation*}
and the right hand side goes to $0$ in view of  \eqref{eq:conv:tilde-eta-k:LpVp}.
Thus we have a subsequence such that $\partial_t  
\bogov [ \div \widetilde{\bd{\eta}}_\pmA ] \rightharpoonup 0$ weakly in $L^2 ( (0,T) , L^2 ( \Omega )^3 )$.
Finally, we have shown that: 
\begin{align*}
\bogov [ \div \widetilde{\bd{\eta}}_\pmA ] 
& \xrightarrow[\pmA\rightarrow \infty]{}  0 
\quad\text{ strongly in } L^p ( (0,T) , \Vp( \Omega )^3 ) , \\
\partial_t  
\bogov [ \div \widetilde{\bd{\eta}}_\pmA ] 
& \xrightharpoonup[\pmA\rightarrow \infty]{} 0
\quad\text{ weakly in } L^2 ( (0,T) , L^2( \Omega )^3 )  .
\end{align*}
\end{enumerate}
This concludes the proof.

\section*{Acknowledgement} 
We would like to thank Céline Grandmont for interesting discussions and comments. 
V. C. and \v S. N.  have been supported by  Præmium Academiæ of \v S. Ne\v casov\' a.  The Institute of Mathematics CAS is supported by RVO:67985840.

\printbibliography

\vspace{1cm}

\end{document}